\documentclass[11pt]{article}
\usepackage[utf8]{inputenc}
\usepackage[T1]{fontenc}
\usepackage{authblk}

\usepackage{eurosym}
\usepackage{amsmath,amssymb,amsthm,amsfonts}
\usepackage{epstopdf}
\usepackage{subcaption}
\usepackage{appendix}
\usepackage{amsmath}
\usepackage{mathrsfs}
\usepackage{amsfonts}
\usepackage{amssymb}
\usepackage{multirow}
\usepackage{rotating}
\usepackage{xcolor}
\usepackage{lscape}
\usepackage{verbatim}
\usepackage{graphicx}
\usepackage{cancel}
\usepackage[normalem]{ulem}
\usepackage{url}
\usepackage{booktabs}
\usepackage{bbm}
\numberwithin{equation}{section} 

\newtheorem{theorem}{Theorem}[section]
\newtheorem{corollary}{Corollary}[section]

\newtheorem{proposition}{Proposition}[section]
\newtheorem{definition}{Definition}[section]
\newtheorem{remark}{Remark}
\newtheorem{example}{Example}

\def\XX{\boldsymbol{X}}
\def\xx{\boldsymbol{x}}

\def\E{\mathbb{E}}

\def\rr{\boldsymbol{r}}

\def\aa{\boldsymbol{a}}

\def\UU{\boldsymbol{U}}

\def\PP{\mathbb{P}}

\def\VV{\boldsymbol{V}}

\def\ss{\boldsymbol{s}}

\def\uu{\boldsymbol{u}}
\def\ii{\boldsymbol{i}}

\def\I{\mathscr{I}}
\def\III{\mathcal{I}}
\def\HHH{\mathcal{H}}
\def\FFF{\mathcal{F}}

\def\PPP{\mathcal{P}}
\def\BBB{\mathcal{B}}
\def\CCC{\mathcal{C}}

\def\YY{\boldsymbol{Y}}
\def\yy{\boldsymbol{y}}

\def\RR{\mathbb{R}}

\def\ZZ{\boldsymbol{Z}}

\def\zz{\boldsymbol{z}}

\def\ff{\boldsymbol{f}}

\def\design{\mathcal{X}_d}

\def\BS{\mathcal{SB}_d}
\def\BP{\mathcal{PB}_d}

\DeclareMathOperator{\rank}{rank}

\begin{document}
	
	\author[1]{Alessandro Mutti\thanks{alessandro.mutti@polito.it}}
	\author[1]{Patrizia Semeraro}
	
	\affil[1]{\small \textit{Department of Mathematical Sciences ``G.L. Lagrange'', Politecnico di Torino, Torino, Italy}}

	\title{Symmetric Bernoulli distributions and minimal dependence copulas}
	\maketitle

	\begin{abstract}
		
		The key result of this paper is to characterize all the multivariate symmetric Bernoulli distributions whose sum is minimal under convex order. 
		In doing so, we automatically characterize extremal negative dependence among Bernoulli random vectors, since multivariate distributions with minimal convex sums are known to be strongly negative dependent. Moreover, beyond its interest per se, this result provides insight into negative dependence within the class of copulas. In particular, two classes of copulas can be built from multivariate symmetric Bernoulli distributions: extremal mixture copulas and FGM copulas. We analyze the extremal negative dependence structures of copulas corresponding to symmetric Bernoulli random vectors with minimal convex sums and explicitly find a class of minimal dependence copulas. Our main results derive from the geometric and algebraic representations of multivariate symmetric Bernoulli distributions, which effectively encode key statistical properties.

	\end{abstract}
	
	\section{Introduction}
	
	A problem extensively studied in applied probability is finding bounds for sums $S=X_1+\dots+X_d$ of random variables with joint distribution in a given Fréchet class $\mathcal{F}_d(F_1,\ldots,F_d)$, i.e.\@ the class of all the joint distributions with one-dimensional $i$-th marginal distribution $F_i$ (see e.g.\@ \cite{denuit1999stochastic}, \cite{embrechts2006bounds}, \cite{lux2017improved}, \cite{ puccetti2013sharp}, \cite{wang2013bounds}). 
	In the fields of insurance and finance, the concept of convex order plays a crucial role since it is a stochastic order that allows the comparison of risks to determine which is lower.
	The problem of finding the upper bound is solved: the upper bound is reached when the risks are comonotonic and their joint distribution is the upper Fréchet bound, that is the maximum element of $\mathcal{F}_d(F_1,\ldots,F_d)$ in concordance order (\cite{kaas2000upper}). 
	The problem of finding the lower bound is not as straightforward: in dimension two the solution is the lower Fréchet bound, but if $d\geq3$ in general it fails to be a distribution (see \cite{joe1997multivariate}). 
	The problem of finding distributions in $\mathcal{F}_d(F_1,\ldots,F_d)$ corresponding to minimal aggregate risk is as yet unsolved in general and this is the problem we focus on.
	Following \cite{puccetti2015extremal}, we call the random vectors with minimal convex sums and their distributions the $\Sigma_{cx}$-smallest elements in $\mathcal{F}_d(F_1,\ldots,F_d)$, when they exist. 
	
	We consider two Fréchet classes. 
	One is the class $\BS$ of multidimensional distributions with one-dimensional Bernoulli marginals with mean $p=\tfrac{1}{2}$, called multivariate symmetric Bernoulli distributions. 
	Multivariate  Bernoulli distributions and their properties are widely investigated in the statistical literature (see e.g.\@ \cite{chaganty2006range}, \cite{dai2013multivariate}, \cite{hu2005dependence}, and \cite{marchetti2016palindromic}), because of the importance of binary data in applications.
	The other class is the whole class of copulas, i.e.\@ the class of multivariate distribution functions with one-dimensional uniform marginals \cite{nelsen2006introduction}. 
	Copulas are widely used to represent dependence among risks in insurance and finance. 
	Usually, marginal risks and their dependence structure are modeled separately, since using Sklar's Theorem it is possible to model dependence among risks with any given distribution using copulas, see e.g.\@ \cite{de2004measuring}, \cite{embrechts2009copulas}, and \cite{embrechts2003using}. 
	An application of copulas to financial risk analysis is provided in \cite{zhang2019application}.
	Therefore, characterizing the class of copulas corresponding to minimal aggregate risk is an important step in understanding the dependence structures associated to low aggregate risk.
	
	While not all Fréchet classes admit a $\Sigma_{cx}$-smallest element (see Example 3.1 of \cite{bernard2014risk}), there always exists a multivariate Bernoulli random vector with sum minimal in convex order.
	In the case of symmetric Bernoulli, the probability mass function (pmf) of the minimal sum in convex order has support on the two adjacent points $(d-1)/2$ or $(d+1)/2$, if $d$ is odd, or it is the degenerate pmf with support on $d/2$ if $d$ is even.
	In the literature, there exist approaches to find a $\Sigma_{cx}$-smallest element: it is possible to consider the unique exchangeable solution (e.g.\@ \cite{hu1999dependence}), or non-exchangeable solutions following Theorem 5.2 of \cite{fontana2024high}, or Lemma 3.1 of \cite{bernard2017robust}.
	However, the above-mentioned approaches find trivial solutions in the symmetric Bernoulli case, such as multivariate pmfs with support on two points only.
	Our novel contribution is to solve the problem of finding and characterizing all $\Sigma_{cx}$-smallest elements in the Fréchet class $\BS$.
	In \cite{cossette2025extremal}, the authors show that $\XX$ is a $\Sigma_{cx}$-smallest element in $\BS$ if and only if $\XX$ is $\Sigma$-countermonotonic (the only if implication is true in general, see \cite{puccetti2015extremal}).
	$\Sigma$-countermonotonicity is a multivariate extension of the bivariate countermonotonicity, that is the maximal negative dependence between two random variables (see \cite{lee2017multivariate}). 
	
	As a consequence, the $\Sigma_{cx}$-smallest elements in $\BS$ define a class of extremely negative dependent symmetric Bernoulli random vectors.
	Although these results are of interest per se, they contribute to the study of negative dependence in a more general framework. 
	In fact, they allow us to explicitly characterize a class of $\Sigma$-countermonotonic copulas, i.e.\@ minimal dependence copulas. Extreme negative dependence and its relationship with minimal risk is extensively studied in the context of insurance and finance (see, among  others, \cite{dhaene1999safest}, \cite{embrechts2003using}, \cite{lee2014multidimensional}, and  \cite{lee2017multivariate}).
	In this framework, the theory of copulas provides a useful tool to model dependence and to find distributional bounds for dependent risks (\cite{embrechts2003using}, \cite{ruschendorf1982random}, \cite{wang2013bounds}).
	We consider two classes of copulas that can be built from multivariate symmetric Bernoulli distributions: the extremal mixture copulas (\cite{mcneil2022attainability} and \cite{tiit1996mixtures}), and the Farlie-Gumbel-Morgenstern (FGM) copulas (\cite{blier2022stochastic}). 
	While FGM copulas are in a one to one relationship with the elements of the class $\BS$, the extremal mixture copulas are in a one to one relationship with the subclass of palindromic Bernoulli distributions.
	We study the dependence structure of the copulas corresponding to $\Sigma_{cx}$-smallest Bernoulli distributions in these two classes. 
	
	This paper proves that it is always possible to find a class of extremal copulas---a subclass of the extremal mixture copulas---that are $\Sigma$-countermonotonic.
	This result can be improved when the dimension of the Fréchet class $d$ is even: in this case, the extremal mixture copulas corresponding to $\Sigma_{cx}$-smallest Bernoulli distributions are $\Sigma_{cx}$-smallest elements in the Fréchet class of copulas. 
	Therefore, if $d$ is even, multivariate uniform variables have a minimum risk element.
	In \cite{cossette2024generalized}, the authors prove that the FGM copulas corresponding to the $\Sigma_{cx}$-smallest elements in $\BS$ are $\Sigma_{cx}$-smallest elements of their class, although they are not $\Sigma$-countermonotonic in the whole class of copulas. For this reason we investigate the negative dependence associated to  the  $\Sigma_{cx}$-smallest FGM copulas  employing widely used measures of dependence: Pearson's correlation, Spearman's rho and Kendall's tau.
	
	Our results follow from the geometrical and algebraic representations of the class $\BS$. 
	This class can be represented as a convex polytope (see \cite{fontana2018representation}) whose extremal generators encode relevant statistical properties, such as extremal dependence or distributional bounds for relevant risk measures. 
	Although extremal generators can be found in closed form in special classes (see \cite{fontana2021model}) and analytically in low dimension \cite{fontana2018representation}, finding them in high dimension becomes computationally infeasible. 
	For this reason, in \cite{fontana2024high} the authors find a way around this limitation and map the class of multivariate Bernoulli distributions with given mean $p$ into an ideal of points in the ring of polynomials with rational coefficients. 
	Using the results in \cite{fontana2024high}, we find an analytical set of polynomials that generate the class $\BS$ and an analytical set of polynomials that generate the class of palindromic distributions. 
	These last generators are extremal points of the polytope and they are associated to the extremal copulas. 
	These connections allow us to find the $\Sigma$-countermonotonic extremal copulas.
	Indeed, the effectiveness of the algebraic representation is that the polynomial coefficients can be used to construct multivariate Bernoulli distributions with given statistical properties.

	The paper is organized as follows. 
	Section~\ref{sec:SymmetricBernoulli} recalls the geometrical representation of $\BS$ and presents the results necessary for our study. The proofs are technical and are left in Appendix~\ref{app:ProofFirstSection}, while some additional comments and complements are in Appendix~\ref{app:ExamplesComplements}.
	Section~\ref{sec:minCX_ExtremalNegDep} introduces the notions of extremal negative dependence and characterizes the $\Sigma_{cx}$-smallest elements of $\BS$ and the extremal negative dependence in the two classes of extremal mixture and FGM copulas.
	Furthermore, in this Section, we find a family of $\Sigma$-countermonotonic copulas. 
	The proofs based on the algebraic representation are in Appendix~\ref{app:ProofSecondSection}. 
	Section~\ref{sec:DependenceMeasures} studies pairwise negative dependence measures and correlation in $\BS$ and in the two classes of copulas. 
	Concluding remarks are given in Section~\ref{Concl}.

	\section{Algebraic representation} \label{sec:SymmetricBernoulli}
	This section introduces an algebraic representation of multivariate Bernoulli random variables that is effective in studying the statistical properties of Bernoulli random vectors. 
	The material in this section is necessary to prove our results on negative dependence in the class $\BS$ and therefore in the class of copulas. 
	Since most of the content of this section is technical, for clarity and readability, the proofs are provided in Appendix~\ref{app:ProofFirstSection}.

	We assume that vectors $\xx = (x_1,\dots,x_d)$ are column vectors and we denote by $A^{\top}$ the transpose of a matrix $A$.
	Given a Bernoulli random vector $\XX=(X_1, \dots, X_d)$ with pmf $f\colon \{0,1\}^d \to [0,1]$, $f \in \BS$, we denote by $\ff = (f_1,\ldots, f_{2^d})$ the vector that contains the values and $f$ over $\design=\{0, 1\}^d$, i.e.\@ $\ff := (f(\xx):\xx \in \design)$.
	We make the non-restrictive hypothesis that the set $\design$ of $2^d$ binary $d$-dimensional vectors is ordered according to the reverse-lexicographical criterion.
	For example, for $d=3$, we have $\mathcal{X}_3=\{000, 100, 010, 110, 001, 101, 011, 111\}$.
	Given that a pmf $f \in \BS$ uniquely determines a vector $\ff$ (and vice versa), we will use the term pmfs to denote also the vectors $\ff$. 
	By $\XX \in \BS$ and $\ff \in \BS$, we mean that the random vector $\XX$ has pmf $f \in \BS$.
	
	Given two matrices $A\in \mathcal{M}(n\times m)$ and $ B\in \mathcal{M}(d\times l)$: 
	\begin{itemize}
		\item if $n=d$, $A||B$ denotes the row concatenation of $A$ and $B$;
		\item  if $m=l$, $A//B$ denotes the column concatenation of $A$ and $B$;
	\end{itemize}
	Finally, we denote by $P(\zz)=\sum_{\ii\in \mathcal{X}_{d-1} }a_{\ii}\zz^{\ii}$ a polynomial in the ring $\mathbb{Q}[\zz]$ of polynomials with rational coefficients in the variables $\zz=(z_1,\ldots,z_{d-1})$, where $\zz^{\ii} = \prod_{j=1}^{d-1} z_j^{i_j}$. 
	To simplify the notation we write $a_{i_1\dots i_{d-1}}:=a_{(i_1,\ldots, i_{d-1})}=a_{\ii}$.
	
	In \cite{fontana2018representation}, the authors show that $\BS$ is a convex polytope, that is
	\begin{equation} \label{polytope}
		\BS = 
		\bigg\{ 
		\ff \in \RR^{2^d} : H_d \ff = \boldsymbol{0}, f_j \ge 0, \sum_{j=1}^{2^d} f_j = 1 
		\bigg\},
	\end{equation}
	where $H_d$ is a $d \times 2^d$ matrix whose rows are $(\boldsymbol{1}_{2^d}-2 \xx_h)^{\top}$, $h \in \{1,\dots,d\}$, where $\boldsymbol{1}_{2^d}$ is the $2^d$-vector with all elements equal to 1, and $\xx_h$ is the $2^d$-vector that contains the $h$-th components of all the $d$-vectors $\xx \in \design$. 
	Therefore, $\BS$ is the convex hull of a finite set of points $\rr_k \in \BS$, $k = 1,\dots,n_d$, called extremal points or extremal pmfs.
	In other terms, for any $\ff \in \BS$, there exist $n_d$ positive weights $\lambda_1, \ldots, \lambda_{n_d}$ summing up to one such that
	\begin{equation*}
		\ff = \sum_{i=1}^{n_d} \lambda_i \rr_i.
	\end{equation*}
	When the dimension $d$ is sufficiently small, the extremal points of the convex polytope can be found using 4ti2 (see \cite{fontana2018representation}).
	However, this representation has computational limitation.
	When the dimension $d$ increases, due to the growth of the number $n_d$, finding all the extremal pmfs becomes computationally infeasible. 
	For example, for the middle-size case $d=6$, the class $\mathcal{SB}_6$ has $n_6 = 707,264$ extremal points.
	To overcame this limitation, the authors of \cite{fontana2024high} introduce a new algebraic representation of any Fréchet class of joint Bernoulli distributions with the same one-dimensional marginals with common mean $p \in (0,1)\cap\mathbb{Q}$, that proves to be extremely effective in the study of the case $p=\tfrac{1}{2}$, i.e.\@ the class  $\BS$. 
	
	Following \cite{fontana2024high}, we define the linear map $\mathcal{H}$ from the class $\BS$ to the polynomial ring with rational coefficients $\mathbb{Q}[z_1,\ldots, z_{d-1}]$ as:
	\begin{align}
		\HHH : \quad \BS &\to\mathbb{Q}[z_1,\ldots, z_{d-1}] \notag 
		\\
		\ff &\mapsto \HHH(\ff) = \sum_{\ii \in \mathcal{X}_{d-1}} a_{\ii} \zz^{\ii}, \label{eq:mapH}
	\end{align}
	where, for every $\ff \in \BS$, the vector of coefficients $\boldsymbol{a} = (a_1,\dots,a_{2^{d-1}})$ in~\eqref{eq:mapH} is given by
	\begin{equation} \label{eq:poly_coeff}
		\boldsymbol{a} = (a_{\ii})_{\ii \in \mathcal{X}_{d-1}} = Q \ff,
	\end{equation}
	with $Q = (I(2^{d-1}) || \tilde{I}(2^{d-1}))$, where $I(2^{d-1})$ is the identity matrix of order $2^{d-1}$ and $\tilde{I}(2^{d-1})$ is the square matrix of order $2^{d-1}$ with $-1$ on the anti-diagonal and 0 elsewhere. 
	For every $\ii \in \mathcal{X}_{d-1}$, set $\ss_{\ii} := (\ii//0) = (i_1,\ldots,i_{d-1},0)$. 
	Because of the form of the matrix $Q$ in~\eqref{eq:poly_coeff}, we can write the image of $\ff$ through $\HHH$ as:
	\begin{equation} \label{eq:mapH2}
		\HHH(\ff) = \sum_{\ii \in \mathcal{X}_{d-1}} (f(\ss_{\ii}) - f(\boldsymbol{1}_d -\ss_{\ii}) ) \zz^{\ii}.
	\end{equation}
	
	We call $\mathcal{C_H}$ the image of $\BS$ through $\HHH$. From Theorem 3.1 in \cite{fontana2024high}, $\mathcal{C_H}\subseteq\III_{\PPP}$, where $\III_{\PPP}\subseteq \mathbb{Q}[\zz]$ is the ideal of polynomials that vanish at points $\PPP = \{\boldsymbol{1}_{d-1}, \boldsymbol{1}_{d-1}^{-j}, j=1,\ldots,d-1 \}$, where $\boldsymbol{1}_{d-1}^{-j}$ is a vector of length $d-1$ with $-1$ in position $j$ and $1$ elsewhere.
	
	Example~\ref{ex:polynomial_d3} in Appendix~\ref{app:ExamplesComplements} provides a simple example of the polynomial representation for $d=3$ and shows that the map $\HHH$ is not injective.
	Indeed, the authors of \cite{fontana2024high} find a basis of the kernel $\mathcal{K}(\HHH)$ of the map, i.e.\@ a basis of the set of pmfs $\ff$ such that $\HHH(\ff)=0$. 
	A basis of the kernel is the set:
	\begin{equation} \label{basis_kernel}
		\begin{split}
			\mathcal{B}_K 
			&= \big\{ 
			f \in \BS : \exists \xx \in \design \text{ such that } f(\xx)=f(\boldsymbol{1}_d - \xx) = \tfrac{1}{2}
			\big\}
			\\
			&= \big\{ 
			\big(\tfrac{1}{2},0,0,\ldots,0,0,\tfrac{1}{2} \big); 
			\big(0,\tfrac{1}{2},0,\ldots,0,\tfrac{1}{2},0 \big); 
			\big(0,0,\tfrac{1}{2},\ldots,\tfrac{1}{2},0,0 \big); 
			\ldots \big\}.
		\end{split}
	\end{equation}
	Notice that if $\ff \in \mathcal{B}_K$ it has support on two points.
	We denote by $\BP$ the class of $d$-dimensional palindromic Bernoulli pmfs, i.e.\@ the pmfs $f$ of Bernoulli random vectors such that $f(\xx) = f(\boldsymbol{1}_d - \xx)$, for every $\xx \in \design$ (see \cite{marchetti2016palindromic} as a reference for palindromic distributions). 
	The proofs of the following propositions are straightforward, yet the results are important to our purposes, because palindromic Bernoulli distributions generate the class of extremal mixture copulas (\cite{mcneil2022attainability}), which are one of our objects of study. 
	
	\begin{proposition} \label{prop:kernel_palin}
		The kernel of the map $\HHH$ coincides with the set of palindromic Bernoulli distributions, $\mathcal{K}(\HHH) \equiv \BP$.
	\end{proposition}
	
	\begin{proposition} \label{prop:extremal_pmfs_kernel}
		The pmfs of the basis $\mathcal{B}_K$ of the kernel $\mathcal{K}(\HHH)$ in~\eqref{basis_kernel} are extremal points of the polytope $\BS$.
	\end{proposition}
	
	The basis $\mathcal{B}_K$ has $2^{d-1}$ pmfs; therefore, there are $2^{d-1}$ extremal points of $\BS$ that have null polynomial.
	The kernel $\mathcal{K}(\HHH)$ is now fully characterized. 
	It is more challenging to characterize the counter-image of a non-null polynomial. 
	In \cite{fontana2024high}, the authors suggest an algorithm to find a particular distribution from a given polynomial, which they call the type-0 pmf. 
	The algorithm adapted to the class $\BS$ is reported in Table~\ref{type0_algorithm}.
	\begin{table}[tb]
		\centering
		\begin{tabular}{l}
			\toprule
			\textbf{Algorithm 1} 
			\\
			\midrule
			\textbf{Input}: A polynomial $P(\zz) = \sum_{\ii \in \mathcal{X}_{d-1}} a_{\ii} \zz^{\ii} \in \III_{\PPP}$, $P(\zz) \not\equiv 0$.
			\\
			\midrule
			For each $\ii \in \mathcal{X}_{d-1}$: 
			\\
			\quad if $a_{\ii} \ge 0$, then $f^P(\ss_{\ii}) = a_j$ and $f^P(\boldsymbol{1}_d - \ss_{\ii}) = 0$; 
			\\
			\quad if $a_{\ii} < 0$, then $f^P(\ss_{\ii}) = 0$ and $f^P(\boldsymbol{1}_d - \ss_{\ii}) = -a_j$. 
			\\
			Normalize $\ff^P$ getting, with a small abuse of notation, $\ff^P = \ff^P / (\sum_{\xx \in \design} f^{P}(\xx))$.
			\\
			\midrule
			\textbf{Output}: A pmf $\ff^P = (f^P_1,\ldots,f^P_{2^d}) \in \BS$ with $\HHH(\ff^P) = P(\zz)$.
			\\
			\bottomrule
		\end{tabular}
		\caption{\emph{Algorithm to generate the type-0 probability mass function associated to $P(\zz)$.}}
		\label{type0_algorithm}
	\end{table}
	
	We conclude this section by presenting two results on the algebraic representation of $\BS$ that are necessary to prove Proposition~\ref{characterization_minCX}, one of our main results on negative dependence in the class $\BS$.
	A preliminary Definition~\ref{def:EquivalentPolynomial} is necessary to introduce the concept of equivalence between polynomials of the ideal $\III_{\PPP}$.
	\begin{definition} \label{def:EquivalentPolynomial}
		Two polynomials $P(\zz)$ and $Q(\zz)$ of the ideal $\III_{\PPP}$ are equivalent, denoted by $P(\zz) \simeq Q(\zz)$, if there exists a constant $\mu > 0$, $\mu\in \mathbb{Q}$, such that $P(\zz) = \mu Q(\zz)$. 
		We denote by $[P(\zz)] = \{ Q(\zz) \in \III_{\PPP}: Q(\zz) \simeq P(\zz) \}$ the set of all the polynomials equivalent to $P(\zz)$.
	\end{definition}
	
	\begin{proposition} \label{prop:EquivalentPolynomials}
		Two equivalent polynomials generate the same type-0 pmf.
	\end{proposition}
	
	Given a polynomial $P(\zz)$, Proposition~\ref{prop:inverseMap} characterizes the set $\HHH^{-1}[P(\zz)] := \big\{\ff \in \BS : \HHH(\ff) \in[ P(\zz)] \big\}$, that is the set of all the pmfs mapped by $\mathcal{H}$ in a polynomial equivalent to $P(\zz)$.
	This proposition is crucial for finding all the extremal negative dependent Bernoulli random vectors, which will be characterized through their polynomials.
	\begin{proposition} \label{prop:inverseMap}
		Consider a polynomial $P(\zz) = \sum_{\ii \in \mathcal{X}_{d-1}} a_{\ii} \zz^{\ii} \in \III_{\PPP}$, such that $P(\zz) \not\equiv 0$.
		Then,
		\begin{equation*}
			\HHH^{-1}[P(\zz)] = \big\{\ff\in \BS: \ff = \lambda \ff^P + (1-\lambda) \ff^K \text{, with } \ff^K \in \mathcal{K}(\HHH), \lambda \in (0,1] \big\},
		\end{equation*}
		where $\ff^P$ is the type-0 pmf of $P(\zz)$.
	\end{proposition}
	
	Finally, the following Proposition highlights the importance of the type-0 pmfs and their link with the generators of $\BS$ as a convex polytope.
	\begin{proposition}\label{prop:extipe0}
		Every extremal point of $\mathcal{SB}_d$ is either a type-0 pmf or an element of $\mathcal{B}_K$.
	\end{proposition}
	
	Remark~\ref{rmk:OtherMarginalMeans} in Appendix~\ref{app:ProofFirstSection} shows that Proposition~\ref{prop:EquivalentPolynomials} also holds for any Fréchet class of joint Bernoulli distributions with common marginals with mean $p$,  $p\in[0,1]$.
	Example~\ref{ex:Ec-im} in Appendix~\ref{app:ExamplesComplements} instead shows that Proposition~\ref{prop:inverseMap} and Proposition~\ref{prop:extipe0} hold for $p=\tfrac{1}{2}$ only, i.e in the class $\BS$.

	\section{Minimal convex sums and extremal negative dependence} \label{sec:minCX_ExtremalNegDep}
	
	In this section, we recall the main ingredients of negative dependence and study the links between extremal negative dependence and minimality in convex order.
	When studying negative dependence and, in particular, extremal negative dependence, the starting point is the definition of countermonotonicity.
	\begin{definition}
		A bivariate random vector $(X,Y)$ is said to be countermonotonic if
		\begin{equation*}
			\PP[(X_1-X_2)(Y_1-Y_2) \leq 0] = 1,
		\end{equation*}
		where $(X_1,Y_1)$ and $(X_2,Y_2)$ are two independent copies of $(X,Y)$.
	\end{definition}
	Although this definition provides a clear characterization of extremal negative dependence for Fréchet classes of dimension $d=2$, there is no unique and straightforward generalization of this concept to dimensions higher than two.
	Various approaches have been proposed to define notions of minimal dependence in Fréchet classes of dimensions higher than two. 
	These notions are known as extremal negative dependence concepts, see \cite{puccetti2015extremal}.
	An intuitive generalization of countermonotonicity is the notion of pairwise countermonotonicity, recently studied in \cite{lauzier2023pairwise}.
	\begin{definition} \label{def:PairwiseCountermonotonic}
		A random vector $\YY = (Y_1,\dots,Y_d)$ is pairwise countermonotonic if the pair $(Y_{j_1},Y_{j_2})$ is countermonotonic, for every $j_1,j_2 \in \{1,\dots,d\}$, $j_1 \neq j_2$.
	\end{definition}
	The distribution function of pairwise countermonotonic random vector in a Fréchet class $\FFF_d(F_1,\dots,F_d)$ is the lower Fréchet bound $F^L_d(x_1,\dots,x_d) = \max(F_1(x_1) + \dots + F_d(x_d)-d+1,0)$.
	However, as discussed in \cite{puccetti2015extremal}, a Fréchet class $\FFF_d(F_1,\dots,F_d)$ admits a pairwise countermonotonic random vector only under very restrictive assumptions on the marginal distributions.
	These requirements were first studied in \cite{dall1972frechet} and are reported in Proposition 3.2 in \cite{puccetti2015extremal}.
	Within the framework of Bernoulli distributions, as discussed in Section 4.1 of \cite{cossette2025extremal}, these conditions imply that a pairwise countermonotonic Bernoulli random vector has marginal means $p_1,\dots,p_d$ such that $p_1+\dots+p_d \leq 1$ or $p_1+\dots+p_d \geq d-1$.
	Also, if $F_j$, $j \in \{1,\dots,d\}$ is a continuous distribution, then the Fréchet class $\FFF_d(F_1,\dots,F_d)$ does not admit any pairwise countermonotonic random vector.
	Therefore, the two Fréchet classes we focus on in this paper, i.e.\@ $\BS$ and the Fréchet class of distributions with standard uniform marginals, do not admit a pairwise countermonotonic random vector, in any dimension $d > 2$. 
	For this reason, we turn our attention to different notions of extremal negative dependence that are based on less restrictive assumptions.
	
	We consider three notions of extremal negative dependence: minimality in convex sums, joint mixability, and $\Sigma$-countermonotonicity.
	Minimality in convex sums consists of finding vectors $\YY$ such that $\sum_{j=1}^d Y_j$ is minimal in convex order in a given class of distributions.
	The convex order is a variability order, thus a random variable that is minimal in convex order is a minimal risk random variable. 
	Therefore, the purpose of this extremal negative dependence is to minimize the aggregate risk. 
	We formally introduce the convex order.
	\begin{definition}
		Given two random variables $Y_1$ and $Y_2$ with finite means, $Y_1$ is said to be smaller than $Y_2$ under the convex order (denoted $Y_1 \le_{cx} Y_2$) if $E[\phi(Y_1)]\leq E[\phi(Y_2)]$, for all real-valued convex functions $\phi$ for which the expectations are finite.
	\end{definition}
	We can now define a class of vectors minimal in convex order, and we call them $\Sigma_{cx}$-smallest elements.
	\begin{definition} \label{def:SigmaCX}
		A $\Sigma_{cx}$-smallest element in a class of distributions $\FFF$ is a random vector $\YY = (Y_1,\dots,Y_d)$ with distribution in $\FFF$ such that the sum of its components are minimal under the convex order, i.e.
		\begin{equation*}
			\sum_{j=1}^d Y_j \leq_{cx} \sum_{j=1}^d Y'_j,
		\end{equation*}
		for any random vector $\YY'$ with distribution in $\FFF$.
	\end{definition}
	
	\begin{remark}\label{remark:sm}
		A desirable property of extremal negative dependence is to minimize a dependence order. 
		Indeed, a pairwise countermonotonic random vector $\YY$ is minimal in supermodular order, i.e.\@ it is such that $\E[\psi(\YY)] \leq \E[\psi(\YY')]$, for any random vector $\YY'$ with the same marginal distributions, and for all supermodular functions $\psi$, such that the expectations are finite.
		We recall that a supermodular function is a function $\psi : \RR^d \to \RR$ such that $\psi(\xx) + \psi(\yy) \leq  \psi(\xx \wedge \yy) + \psi(\xx \vee \yy)$, for all $\xx,\yy \in \RR^d$.
		As discussed in \cite{puccetti2015extremal}, instead of considering all supermodular functions, we consider the subclass of supermodular functions such that $\psi(\xx) = \phi(x_1+\dots+x_d)$, for some convex function $\phi:\RR \to \RR$.
		From this perspective, the definition of $\Sigma_{cx}$-smallest elements arises naturally, although in general they are not minimal in supermodular order. 
		If we restrict to exchangeable Bernoulli random vectors, we have a particular case, where $\Sigma_{cx}$-smallest elements are minimal in supermodular order, as proved in \cite{frostig2001comparison}.
	\end{remark}
	
	The next notion we present is closely related to the previous definition of $\Sigma_{cx}$-smallest elements. It is the joint mixability property and it has been introduced in \cite{wang2011complete}.
	\begin{definition} \label{def:JointMix}
		A $d$-dimensional random vector $\YY = (Y_1,\ldots,Y_d)$ is said to be a joint mix if 
		\begin{equation*}
			\PP \bigg( \sum_{j=1}^d Y_j = k \bigg) = 1,
		\end{equation*}
		for some $k \in \RR$, called joint center.
	\end{definition}
	Since any joint mix minimizes the variance of the sum of its components, it is obvious that a joint mix is also a $\Sigma_{cx}$-smallest element of its Fréchet class, assuming that it has marginals with finite mean.
	
	However, there exist Fréchet classes that do not admit $\Sigma_{cx}$-smallest elements or joint mixes.
	Therefore, we conclude this section with the last notion we consider, the $\Sigma$-countermonotonicity property, introduced in \cite{puccetti2015extremal}.
	This definition is significant because every Fréchet class admits a $\Sigma$-countermonotonic random vector.
	\begin{definition} \label{def:SigmaCountermonotonic}
		A $d$-dimensional random vector $\YY = (Y_1,\ldots,Y_d)$ is $\Sigma$-countermonotonic if, for every subset $J \subseteq \{1,\ldots,d\}$, the pair $(\sum_{j \in J} Y_j,\sum_{j \notin J} Y_j)$ is countermonotonic.
	\end{definition}
	We use the convention $\sum_{j \in \emptyset} Y_j = 0$.
	In \cite{puccetti2015extremal}, the authors show that, in Fréchet classes where pairwise countermonotonicity is admissible, a random vector is pairwise countermonotonic if and only if it is $\Sigma$-countermonotonic.
	Instead, if a Fréchet class admits a joint mix or a $\Sigma_{cx}$-smallest pmf, they are always $\Sigma$-countermonotonic.
	
	In Section~\ref{sec:BernoulliEtremalDep}, we develop the study of extremal negative dependence within the class $\BS$, while the discussions in the class of extremal mixture copulas and in the class of FGM copulas are presented in Section~\ref{sec:EMcopulas} and Section~\ref{sec:FGMcopulas}, respectively.

	\subsection{Symmetric Bernoulli distributions} \label{sec:BernoulliEtremalDep}
	
	An important result about extremal negative dependence notions within the class $\BS$ is Theorem 4.1 in \cite{cossette2025extremal}.
	It states that a Bernoulli random vector is $\Sigma$-countermonotonic if and only if it is a $\Sigma_{cx}$-smallest element in its Fréchet class.
	In fact, every Fréchet class with Bernoulli-distributed marginals admits a $\Sigma_{cx}$-smallest element (see the construction in Lemma 3.1 of \cite{bernard2017robust}).
	
	Our main result is to completely characterize the class of $\Sigma_{cx}$-smallest elements in $\BS$ through the algebraic representation discussed in Section~\ref{sec:SymmetricBernoulli}. 
	This leads to a complete characterization of $\Sigma$-countermonotonic random vectors in $\BS$.
	
	The problem to find $\Sigma_{cx}$-smallest elements is trivial if we restrict the analysis to exchangeable Bernoulli random vectors with marginal mean $p$, for any $p\in (0,1)$. 
	In this case there is only one $\Sigma_{cx}$-smallest element in the class, and, as already mentioned in Remark~\ref{remark:sm}, it is also minimal in supermodular order. The general problem, even with common marginal means $p$, is still open.
	In \cite{fontana2024high}, using the algebraic representation of multivariate  Bernoulli pmfs of Bernoulli random vectors with common means $p$, Theorem~5.2 provides an algorithm to find a not exchangeable $\Sigma_{cx}$-smallest element in the class. 
	If $p=\frac{1}{2}$, we now prove a stronger result, since we explicitly characterize all of them.
	Due to their technical nature, the proofs of this Section are given in Appendix~\ref{app:ProofSecondSection}.
	
	Given $d$, we set two integers $M_d$ and $m_d$ such that $M_d = m_d = \frac{d}{2}$, if $d$ is even, and $M_d = \frac{d-1}{2}$ and $m_d = \frac{d+1}{2}$, if $d$ is odd.
	Then, we define
	\begin{equation} \label{eq:ChiStar}
		\mathcal{X}_d^{\ast} = 
		\left\{ 
		\xx \in \mathcal{X}_d : \sum_{h=1}^{d} x_h = M_d \quad \text{or} \quad \sum_{h=1}^{d} x_h = m_d 
		\right\},
	\end{equation}
	as the set of $d$-dimensional binary vectors with sum of the components equal to $M_d$ or to $m_d$, and
	\begin{equation*}
		\I_{d-1}^{\ast} = 
		\left\{ 
		\ii \in \mathcal{X}_{d-1} : \sum_{h=1}^{d-1} i_h = M_d \quad \text{or} \quad \sum_{h=1}^{d-1} i_h = m_d 
		\right\},
	\end{equation*}
	as the set of $(d-1)$-dimensional binary vectors with sum of the components equal to $M_d$ or to $m_d$.
	The following Proposition is a restatement of Proposition~5.2 in \cite{fontana2024high}.
	\begin{proposition} \label{prop:supportBernoulliSumCX}
		The $\Sigma_{cx}$-smallest pmfs in $\BS$ have support entirely contained in $\mathcal{X}_d^{\ast}$.
	\end{proposition}

	It is worth noting that when $d$ is even, a Bernoulli random vector is a $\Sigma_{cx}$-smallest element in $\BS$ if and only if it is a joint mix.
	Instead, when $d$ is odd, there does not exist any joint mix in $\BS$.
	Therefore, when joint mixability is supported, i.e.\@ when $d$ is even, the definitions of $\Sigma$-countermonotonic random vector, $\Sigma_{cx}$-smallest element, and joint mix coincide.
	Therefore, building on Proposition~\ref{prop:supportBernoulliSumCX}, we can characterize extremal negative dependence, by considering the class of $\Sigma_{cx}$-smallest elements.
	We first identify the set of $\Sigma_{cx}$-smallest pmfs in $\mathcal{K}(\mathcal{H})$.
	\begin{proposition} \label{prop:minCX_kernel}
		Let $\ff^{K\ast}\in \mathcal{K}(\mathcal{H})$ be a $\Sigma_{cx}$-smallest pmf.
		Therefore, $\ff^{K\ast}$ is the convex linear combination of the $\Sigma_{cx}$-smallest elements of the basis of the kernel $\mathcal{B}_K$ in~\eqref{basis_kernel}.
	\end{proposition}
	
	The $\Sigma_{cx}$-smallest pmfs in $\mathcal{B}_K$ are easy to identify, since they have support on two points: $\xx^{(1)}$, such that $\sum_{j=1}^{d} x_j^{(1)} = M_d$, and $\boldsymbol{1}_d - \xx^{(1)}$.
	We now consider the entire class $\BS$.
	The following Theorem characterizes the coefficients of the polynomials corresponding to the $\Sigma_{cx}$-smallest pmfs of the class $\BS$.
	\begin{theorem} \label{thm:polyminCX}
		Let $\ff \in \BS$ be a $\Sigma_{cx}$-smallest element in $\BS$. Then, the coefficient of the polynomial $\mathcal{H}(\ff) = P(\zz) = \sum_{\ii \in \mathcal{X}_{d-1}} a_{\ii} \zz^{\ii} \in \III_{\PPP}$ are such that:
		\begin{enumerate}
			\item \label{null_coeff} $a_{\ii} = 0$ for every $\ii \notin \I^{\ast}_{d-1}$;
			\item \label{sum_constants} The sum of the coefficients of the monomials of the same order is equal to $0$;
			\item \label{sum_linear} The sum of the coefficients of the monomials with $z_j$ is equal to $0$, for every $j \in \{1,\ldots,d-1\}$.
		\end{enumerate}
	\end{theorem}
	
	From Theorem~\ref{thm:polyminCX}, all the polynomials $P^{\ast}(\zz)$ with a $\Sigma_{cx}$-smallest pmf in their counter-image $\HHH^{-1}[P^{\ast}(\zz)]$ are of the form:
	\begin{equation} \label{poli_minCX}
		P^{\ast}(\zz) = \sum_{\ii \in \I_{d-1}^{\ast}} a_{\ii} \zz^{\ii},
	\end{equation}
	where the coefficients $a_{\ii}$, $\ii \in \I_{d-1}^{\ast}$, verify Point~\ref{sum_constants} and Point~\ref{sum_linear} of the above Theorem.
	The next Corollary~\ref{corollary_system} to Theorem~\ref{thm:polyminCX} proves that the coefficients of the polynomials of the $\Sigma_{cx}$-smallest pmfs of the class $\BS$ are the solutions of a homogeneous linear system.
	Let ${n}^*_d$ be the number of vectors of $\I_{d-1}^{\ast}$; it is given by:
	\begin{equation*}
		n^*_d = 
		\begin{cases}
			\binom{d-1}{\frac{d-1}{2}} + \binom{d-1}{\frac{d+1}{2}}, &\text{if $d$ is odd} 
			\\
			\binom{d-1}{\frac{d}{2}}, &\text{if $d$ is even}
		\end{cases}.
	\end{equation*}
	
	\begin{corollary} \label{corollary_system}
		If $\ff\in \BS$ is a $\Sigma_{cx}$-smallest element in $\BS$ the coefficients ${\aa}=(a_{\ii}, \ii \in \I^{\ast}_{d-1})$ of the polynomial $\mathcal{H}(\ff)=P(\zz) = \sum_{\ii \in \mathcal{X}_{d-1}} a_{\ii} \zz^{\ii}$ are the solutions of
		\begin{equation} \label{eq:linear_system}
			A_d {\aa} = \boldsymbol{0},
		\end{equation}
		where $A_d$ is obtained from the matrix $A_{\I^*_{d-1}} = (\ii; \ii\in\I_{d-1}^{\ast}) \in {\mathcal{M}((d-1)\times n^*_d)}$, whose columns are the elements $\ii\in\I_{d-1}^{\ast}$.
		In particular,
		\begin{itemize}
			\item if $d$ is even, $A_d=(\boldsymbol{1}_{n^*_d}^{\top}//A_{\I^*_{d-1}}) {\in \mathcal{M}(d\times n^*_d)}$;
			\item if $d$ is odd, $A_d=(R_1//R_2//A_{\I^*_{d-1}}) \in \mathcal{M}((d+1)\times n^*_d)$, where $R_1\in \mathcal{M}(1\times n^*_d)$ is a row with ones in correspondence of the indexes $\ii$ with sum $M_d$ and zeros elsewhere, and $R_2\in \mathcal{M}(1\times n^*_d)$ is a row with ones in correspondence of the indexes $\ii$ with sum $m_d$ and zeros elsewhere.
		\end{itemize}
	\end{corollary}
	Proposition~\ref{prop:MatrixSameRank} in Appendix~\ref{app:ExamplesComplements}  states a general property of the linear systems in~\eqref{eq:linear_system} in two consecutive dimensions.
	
	The following two examples characterizes the polynomials of $\Sigma_{cx}$-smallest pmfs in dimensions $d=3$ and $d=4$, respectively.		
	\begin{example} \label{X_d=3}
		We consider $d=3$. Since $d$ is odd, we have $M_d = \frac{d-1}{2} = 1$ and $m_d = \frac{d+1}{2} = 2$.
		We have  $\I^{\ast}_{2} = \{(1,0),(0,1),(1,1)\}$.
		The first row of $A_3$ is equal to 1 if $i_1+i_2 = M_d = 1$ and 0 otherwise, while the second row is the opposite,
		\begin{equation*}
			A_3 = 
			\begin{pmatrix}
				1 & 1 & 0 \\
				0 & 0 & 1 \\
				1 & 0 & 1 \\
				0 & 1 & 1 \\
			\end{pmatrix}.
		\end{equation*}
		$A_3$ is a $4\times3$ matrix and $\rank(A_3) = 3$.
		Therefore, since the number of unknowns ($a_{10}, a_{01}, a_{11}$) is equal to the rank of the matrix, the linear system in~\eqref{eq:linear_system} admits only the null solution, i.e.\@ $a_{10} = a_{01} = a_{11} = 0$.
		Hence, all the $\Sigma_{cx}$-smallest pmfs in $\mathcal{SB}_3$ have null polynomials.
	\end{example}
	
	\begin{example}\label{X_d=4}
		We consider $d=4$.
		Since $d$ is even, we have $M_d = m_d = \frac{d}{2} = 2$ and $\I^{\ast}_{3} = \{(1,1,0),(1,0,1),(0,1,1)\}$. 
		The first row of $A_4$ is a vector of all ones:
		\begin{equation*}
			A_4 = 
			\begin{pmatrix}
				1 & 1 & 1 \\
				1 & 1 & 0 \\
				1 & 0 & 1 \\
				0 & 1 & 1 \\
			\end{pmatrix}.
		\end{equation*}
		$A_4$ is a $4 \times 3$ matrix and $\rank(A_4) = 3$.
		Therefore, since the number of unknowns ($a_{110},a_{101},a_{011}$) is equal to the rank of the matrix, the linear system in~\eqref{eq:linear_system} admits only the null solution, i.e.\@ $a_{110} = a_{101} = a_{011} = 0$.
		Hence, all the $\Sigma_{cx}$-smallest pmfs in $\mathcal{SB}_4$ have null polynomials.
	\end{example}
	
	\begin{remark}
		As shown in Example~\ref{X_d=3} and Example~\ref{X_d=4}, in the cases $d=3$ and $d=4$, all the $\Sigma_{cx}$-smallest pmfs have null polynomial, i.e.\@ $\HHH(\ff) = 0$, if $\ff$ is $\Sigma_{cx}$-smallest. 
		Therefore, the set of $\Sigma_{cx}$-smallest pmfs is included in $\mathcal{K}(\HHH)$.
		Thus, for $d \le 4$, both the $\Sigma_{cx}$-smallest pmfs and the $\Sigma_{cx}$-maximal pmf (the upper Fréchet bound) are palindromic Bernoulli distributions.
	\end{remark}
	
	Theorem~\ref{thm:polyminCX} ensures that the polynomials of all $\Sigma_{cx}$-smallest pmfs of $\BS$ are solutions of the homogeneous linear system in~\eqref{eq:linear_system}. 
	However, there are pmfs in $\BS$ that are not $\Sigma_{cx}$-smallest elements, but generate a polynomial of the form in~\eqref{poli_minCX}. 
	This is a consequence of the fact that the map $\HHH$ is not injective.
	Proposition~\ref{characterization_minCX} states the key result of this section, because it completely characterizes the class of $\Sigma_{cx}$-smallest pmfs.
	
	\begin{proposition} \label{characterization_minCX}
		Let $P^{\ast}(\zz) \in \III_{\PPP}$ be a non-null polynomial that verifies the three properties of Theorem~\ref{thm:polyminCX}. 
		Then, the type-0 pmf $\ff^{\ast}$ corresponding to $P^{\ast}(\zz)$ is a $\Sigma_{cx}$-smallest pmf of $\BS$ and the set
		\begin{equation*}
			\big\{ \ff : \ff = \lambda \ff^{\ast} + (1-\lambda) \ff^{K\ast}, \lambda \in (0,1] \big\},
		\end{equation*}
		where $\ff^{K\ast}$ is a $\Sigma_{cx}$-smallest pmf with null polynomial, is the set of all $\Sigma_{cx}$-smallest pmfs corresponding to polynomials equivalent to $P^{\ast}(\zz)$.
	\end{proposition}
	
	The algebraic representation and the proofs of the main results are technical; however their strength lies in their simple use. 
	In what follows, we illustrate how to apply these results to find a $\Sigma_{cx}$-smallest element in $\BS$. 
	Then, we show how, at least in principle, our results allow us to find all the $\Sigma_{cx}$-smallest elements.
	To find a $\Sigma_{cx}$-smallest element in the class $\BS$, we have to follow the following steps:
	\begin{enumerate}
		\item Choose a polynomial $P(\zz) = \sum_{\ii \in \mathcal{X}_{d-1}} a_{\ii} \zz^{\ii}$, with coefficients that satisfies the conditions in Theorem~\ref{thm:polyminCX}.
		\item Apply Algorithm~\ref{type0_algorithm} to find the type-0 pmf $\ff^*$.
	\end{enumerate}
	The type-0 pmf $\ff^*$  is $\Sigma_{cx}$-minimal.
	
	To find all the $\Sigma_{cx}$-smallest elements we need Corollary~\ref{corollary_system} that provides an approach to find all the coefficients of the polynomials that satisfy the conditions of Theorem~\ref{thm:polyminCX}. 
	Then, for each polynomial we find the type zero pmf $\ff^*$ and by Proposition~\ref{characterization_minCX}  all the $\Sigma_{cx}$-minimal pmfs are  $\ff = \lambda \ff^{\ast} + (1-\lambda) \ff^{K\ast}, \lambda \in (0,1]$, where $\ff^{K\ast}$ is a convex combination of pmfs in $\mathcal{B}_K$ with support in $\chi^*_d$. 
	We recall that this is equivalent to find all the $\Sigma$-countermonotonic elements in $\BS$, that, when $d$ is even, have also the joint mixability property.

	We conclude this section with Example~\ref{ex:X_d=5} and Example~\ref{ex:X_d=6} that characterize the $\Sigma_{cx}$-smallest elements of the classes $\mathcal{SB}_5$ and $\mathcal{SB}_6$, respectively.
	\begin{example} \label{ex:X_d=5}
		In this example, we show how to find all the $\Sigma_{cx}$-smallest elements of the class $\mathcal{SB}_5$. 	
		The matrix $A_5$ is reported in Example~\ref{ex:build_matrix} in Appendix~\ref{app:ExamplesComplements}.
		Since $\rank(A_5) = 5$, the solution space of the system in~\eqref{eq:linear_system} has dimension ${n}^*_5 - \rank(A_5) = 5$.
		A basis of the space of the solutions of $A_5 \aa = \boldsymbol{0}$  is
		\begin{equation*}
			\begin{split}
				\mathcal{A} 
				= 
				\big\{
				& \aa^{(1)} = (0,1,-1,0,-1,1,0,0,0,0); \\
				& \aa^{(2)} = (0,1,0,-1,-1,0,1,0,0,0); \\
				& \aa^{(3)} = (1,0,-1,0,-1,0,0,1,0,0); \\
				& \aa^{(4)} = (1,0,0,-1,-1,0,0,0,1,0); \\
				& \aa^{(5)} = (1,1,-1,-1,-1,0,0,0,0,1) 
				\big\}.
			\end{split}
		\end{equation*}
		Thus, every polynomial whose coefficients
		\begin{equation*}
			\aa = (a_{1100},a_{1010},a_{0110},a_{1110},a_{1001},a_{0101},a_{1101},a_{0011},a_{1011},a_{0111})
		\end{equation*}
		are a linear combination of the basis $\mathcal{A}$ verify the three assumptions of Theorem~\ref{thm:polyminCX} and its type-0 pmf is a $\Sigma_{cx}$-smallest element in $\mathcal{SB}_5$.
		For example, the polynomial corresponding to the vector $\aa^{(1)}$ is $ P_1(\zz) = z_1 z_3 - z_2 z_3 - z_1 z_4 + z_2 z_4 $ and the corresponding type-0 pmf $\ff^{(1)}$ is such that $f^{(1)}((1,0,1,0,0)) = f^{(1)}((1,0,0,1,1)) = f^{(1)}((0,1,1,0,1)) = f^{(1)}((0,1,0,1,0)) = \frac{1}{4}$ and it is zero elsewhere. 
		Following Lemma 2.3 in \cite{terzer2009large} that gives the conditions for a pmf to be an extremal point, it can be proved that $\ff^{(1)}$ is an extremal pmf of the polytope $\mathcal{SB}_5$.
		A general $\Sigma_{cx}$-smallest pmfs in $\mathcal{H}^{-1}[P_1(\zz)]$ can be found as $\ff = \lambda \ff^{(1)}+(1-\lambda)\ff^{K*}$, where $\ff^{K*}$ is a $\Sigma_{cx}$-smallest of $\mathcal{K}(\HHH)$.
	\end{example}
	
	\begin{example} \label{ex:X_d=6}
		In this example, we show how to find all the $\Sigma_{cx}$-smallest elements of the class $\mathcal{SB}_6$. 
		After ordering the columns of $A_6$ in Example~\ref{ex:build_matrix} according to the reverse-lexicographical order, we find that the basis $\mathcal{A}$ of the solution space in Example~\ref{ex:X_d=5} is also a basis of the space of solutions of $A_6 \aa = \boldsymbol{0}$.
		Thus, every polynomial that have a $\Sigma_{cx}$-smallest pmf in its counter-image has coefficients
		\begin{equation*}
			\aa = (a_{11100},a_{11010},a_{10110},a_{01110},a_{11001},a_{10101},a_{01101},a_{10011},a_{01011},a_{00111})
		\end{equation*}
		that are a linear combination of $\aa^{(1)}, \aa^{(2)}, \aa^{(3)}, \aa^{(4)}$, and $\aa^{(5)}$.
		For example, the polynomial with coefficients $\aa^{(1)}$ is $ P_1(\zz) = z_1 z_2 z_4 - z_1 z_3 z_4 - z_1 z_2 z_5 + z_1 z_3 z_5 $ and the corresponding type-0 pmf $\ff^{(1)}$ is such that $f^{(1)}((1,1,0,1,0,0)) = f^{(1)}((0,1,0,0,1,1)) = f^{(1)}((0,0,1,1,0,1)) = f^{(1)}((1,0,1,0,1,0)) = \frac{1}{4}$ and it is zero elsewhere. 
		As in Example~\ref{ex:X_d=5}, it can be proved that $\ff^{(1)}$ is an extremal pmf of the polytope $\mathcal{SB}_6$.
		Finally, we consider the linear combination $\tilde{\aa} = \aa^{(1)} - \aa^{(2)} - \aa^{(3)} + \aa^{(4)}$. 
		The resulting polynomial is $\tilde{P}(\zz) = P_1(\zz) - P_2(\zz) - P_3(\zz) + P_4(\zz) = z_1 z_3 z_5 - z_2 z_3 z_5 - z_1 z_4 z_5 + z_2 z_4 z_5$ and its type-0 pmf $\tilde{\ff}$ is such that $\tilde{f}((1,0,1,0,1,0)) = \tilde{f}((1,0,0,1,0,1)) = \tilde{f}((0,1,1,0,0,1)) = \tilde{f}((0,1,0,1,1,0)) = \frac{1}{4}$ and zero elsewhere. 
		It can be proved that also $\tilde{\ff}$ is an extremal point of $\mathcal{SB}_6$.
	\end{example}

	\subsection{Extremal negative dependent  copulas} \label{sec:CopulasExtremalDep}
	
	We recall that a $d$-dimensional copula is the restriction to the hypercube $[0,1]^d$ of the cumulative distribution function(cdf) of a $d$-dimensional random vector $\UU$, that is a random vector with standard uniform marginals.
	In the sequel, we identify the copula with the corresponding cdf.
	There are two classes of copulas that can be constructed from symmetric Bernoulli distributions: extremal mixture copulas and Farlie-Gumbel-Morgenstern (FGM) copulas.
	Both of these classes inherit some dependence properties from $\BS$.
	In particular, the results on extremal negative dependence within the Bernoulli class allow us to find a family of extremal copulas that are $\Sigma$-countermonotonic, i.e.\@ they represent extremal negative dependence in the entire class of copulas, and a class of copulas with the joint mixability property.

	\subsubsection{Extremal Mixture Copulas \label{sec:EMcopulas}}
	
	In this section, we study the class of extremal mixture copulas. These copulas are in a one to one correspondence with the palindromic Bernoulli distributions (see \cite{mcneil2022attainability}) that coincides with the kernel of the map $\HHH$ (Proposition~\ref{prop:kernel_palin}).
	
	\begin{definition}
		Given a standard uniform random variable $U$, an extremal copula with index set $J \subseteq \{1,\ldots,d\}$ is the distribution function of the $d$-dimensional random vector $\VV = (V_1, \ldots, V_d)$ where $V_j \overset{d}{=} U$ if $j \in J$, and $V_j \overset{d}{=} 1-U$ if $j \notin J$, for every $j \in \{1,\ldots,d\}$. 
	\end{definition}
	
	For a general dimension $d$, there exist $2^{d-1}$ different extremal copulas. 
	Given $\ii \in \mathcal{X}_{d-1}$, recall that $\ss_{\ii} = (s_{\ii,1},\dots,s_{\ii,d}) := (\ii // 0) = (i_1,\ldots,i_{d-1},0)$ and let $J_{\ii} = \{j \in \{1,\ldots,d\}: s_{\ii,j} = 1 \}$ be the set of indexes corresponding to ones in $\ss_{\ii}$.
	It is possible to infer the explicit form of the copulas, that is, for every $\ii \in \mathcal{X}_{d-1}$:
	\begin{equation*}
		C_{\ii}(\uu) = (\min_{j \in J_{\ii}} u_j + \min_{j \notin J_{\ii}} u_j - 1)^+, \quad \uu \in [0,1]^d,
	\end{equation*}
	where $x^+ = \max(0,x)$; we use the convention $\min_{j \in \emptyset} u_j = 1$.
	
	It is possible to consider a wider class of copulas, by considering mixtures of extremal copulas, see \cite{tiit1996mixtures}.
	\begin{definition}
		An extremal mixture copula $C$ is a copula of the form 
		\begin{equation*}
			C = \sum_{\ii \in \mathcal{X}_{d-1}} w_{\ii} C_{\ii},
		\end{equation*} 
		where, for every $\ii \in \mathcal{X}_{d-1}$, $C_{\ii}$ is the extremal copula with index set $J_{\ii}$ and the weights $w_{\ii}$ are such that $w_{\ii} \ge 0$, for every $\ii \in \mathcal{X}_{d-1},$ and $\sum_{\ii \in \mathcal{X}_{d-1}} w_{\ii} = 1$.
	\end{definition}
	We denote by $\CCC_d^{\text{EM}}$ the class of extremal mixture copulas.
	The following Proposition~\ref{prop:EMsymmBernoulli} has been proved in \cite{mcneil2022attainability} and states that there exists a non-injective map between the class of multivariate Bernoulli distributions and the class of extremal mixture copulas.
	
	\begin{proposition} \label{prop:EMsymmBernoulli}
		Let $U$ be a standard uniform random variable and $\XX$ a $d$-dimensional multivariate Bernoulli random vector with pmf $f$. Let $\XX$ and $U$ be independent. 
		Then the cdf of the uniform random vector 
		\begin{equation} \label{eq:EMdistribution}
			\VV = U\XX + (1-U)(\boldsymbol{1}_d-\XX)
		\end{equation}
		is an extremal mixture copula with weights given by
		\begin{equation} \label{eq:EMweights}
			w_{\ii} = f(\ss_{\ii}) + f(\boldsymbol{1}_d - \ss_{\ii}),
		\end{equation}
		for each $\ii \in \mathcal{X}_{d-1}$.
	\end{proposition}
	
	Given an extremal mixture copula with weights $w_{\ii}$, for every $\ii \in \mathcal{X}_{d-1}$, there exist infinitely many Bernoulli distributions satisfying~\eqref{eq:EMweights}.
	However, it is possible to identify a unique Bernoulli distribution by considering the class of palindromic Bernoulli distributions $\BP$, characterized by the constraint $f(\ss_{\ii}) = f(\boldsymbol{1}_d - \ss_{\ii})$, for every $\ii \in \mathcal{X}_{d-1}$. 
	Therefore, the class $\BP$ is in a one to one correspondence with the family of extremal mixture copulas $\CCC_d^{\text{EM}}$, see \cite{mcneil2022attainability}:
	\begin{equation}\label{eq:iff}
		\BP \longleftrightarrow  \CCC_d^{\text{EM}}.
	\end{equation}
	In particular, the extremal copulas correspond to the pmfs of the basis $\mathcal{B}_{K}$ of $\mathcal{K}(\HHH)$ in~\eqref{basis_kernel}, while an extremal mixture copula corresponds to a convex linear combination of elements of this basis.

	The results of Section~\ref{sec:BernoulliEtremalDep} are useful to explore the concept of negative dependence in the class of extremal mixture copulas.
	We conclude this section with three results within the class $\CCC_d^{\text{EM}}$.
	
	\begin{proposition} \label{prop:EM_ConvexOrder}
		Let $\XX, \XX' \in {\BS}$ and let $\VV$ and $\VV'$ be the corresponding multivariate random vectors defined in~\eqref{eq:EMdistribution}.
		Then, 
		\begin{equation*}
			\sum_{j=1}^d X_j \le_{\text{cx}} \sum_{j=1}^d X'_j \iff  \sum_{j=1}^d V_j \le_{\text{cx}} \sum_{j=1}^d V'_j.
		\end{equation*}
	\end{proposition}
	
	\begin{proof}
		Given a Bernoulli random vector $\XX \in \BS$, from the stochastic representation in~\eqref{eq:EMdistribution}, we have that
		\begin{equation*}
			\sum_{j=1}^d V_j|(U=u) = \sum_{j=1}^d u X_j + (1-u)(1-X_j) = d(1-u) + (2u-1)\sum_{j=1}^d X_j.
		\end{equation*}
		Since $f(ax+b)$, with $a,b \in \RR$, is convex if $f$ is a convex function, it follows that 
		\begin{equation*}
			\sum_{j=1}^d X_j \le_{\text{cx}} \sum_{j=1}^d X'_j \iff \sum_{j=1}^d V_j|(U=u) \le_{\text{cx}} \sum_{j=1}^d V'_j|(U=u).
		\end{equation*}
		From Theorem 3.A.12(b) in \cite{shaked2007stochastic}, the convex order is closed under mixtures and we have the assert.
	\end{proof}

	Proposition~\ref{prop:EM_ConvexOrder} implies that if $\XX\in \BS$ is a $\Sigma_{cx}$-smallest element then $\VV$ in~\eqref{eq:EMdistribution} is a $\Sigma_{cx}$-smallest element in $\CCC_d^{\text{EM}}$. 
	An important consequence of the one to one map in~\eqref{eq:iff} is that we can construct all the extremal mixture copulas from $\BP$.
	Therefore, the $\Sigma_{cx}$-smallest elements in $\CCC_d^{\text{EM}}$ can be constructed from $\Sigma_{cx}$-smallest palindromic Bernoulli random vectors.
	We recall that, Proposition~\ref{prop:minCX_kernel} identifies all the $\Sigma_{cx}$-smallest pmfs in $\BP$. 
	
	The last two results of this section are more important, as they characterize extremal negative dependence in the entire class of copulas, not only in $\CCC_d^{\text{EM}}$.
	Proposition~\ref{prop:Sigma_ctm} states that the extremal copulas built from $\Sigma_{cx}$-smallest Bernoulli random vectors in $\mathcal{B}_K$  are $\Sigma$-countermonotonic, but, in general, are not $\Sigma_{cx}$-smallest in the entire Fréchet class of copulas.
	Proposition~\ref{prop:V_joint_mix}, instead, shows that, when $d$ is even, the extremal mixture copulas corresponding to $\Sigma_{cx}$-smallest pmfs of $\BP$ have the joint mixability property, hence they are $\Sigma_{cx}$-smallest elements in the entire class of copulas.

	\begin{proposition} \label{prop:Sigma_ctm}
		Let $\XX$ be a Bernoulli random vector with pmf $\ff \in \mathcal{B}_K$, where $\mathcal{B}_K$ is the basis of the kernel of $\HHH$, given in~\eqref{basis_kernel}. 
		If $\XX$ is a $\Sigma_{cx}$-smallest element in $\BS$, then the corresponding random vector $\VV$, as defined in~\eqref{eq:EMdistribution}, is $\Sigma_{cx}$-smallest in $\CCC_d^{\text{EM}}$ and  $\Sigma$-countermonotonic.
	\end{proposition}
	\begin{proof}
		Since $\XX$ is a $\Sigma_{cx}$-smallest element in $\BS$ (and also in $\BP$), the random vector $\VV$ is a $\Sigma_{cx}$-smallest element in $\CCC_d^{\text{EM}}$ as a consequence of Proposition~\ref{prop:EM_ConvexOrder}.
		We now prove that $\VV$ is $\Sigma$-countermonotonic.
		Let $I \subset \{1,\ldots,d\}$ be a set of indexes different from the empty set. By hypothesis, $\XX$ has support only on $\xx$ and $\boldsymbol{1}_d-\xx$. 
		Since $\XX$ is $\Sigma_{cx}$-smallest, we know that it has support on points $\xx \in \design^*$, defined in~\eqref{eq:ChiStar}.
		We recall that the sum of the components of $\xx \in \design^*$ is equal to $M_d$ or $m_d$, with $M_d = m_d = d/2$, when $d$ is even, and $M_d=(d-1)/2$ and $m_d=(d+1)/2$, when $d$ is odd. Therefore, we have two alternatives: either $\sum_{j=1}^d x_j = M_d$ and $\sum_{j=1}^d (1-x_j) = m_d$, or $\sum_{j=1}^d x_j = m_d$ and $\sum_{j=1}^d (1-x_j) = M_d$. We consider that case $\sum_{j=1}^d x_j = M_d$, the proof of the other case is the same by setting  $\yy = \boldsymbol{1}_d-\xx$.
		Let $k := \sum_{j \in I} x_j$. We have:
		\begin{equation*}
			\sum_{j \in \bar{I}} x_j = M_d - k, \quad \sum_{j \in I}(1-x_j) = |I|-k, \quad \sum_{j \in \bar{I}}(1-x_j) = m_d-(|I|-k),
		\end{equation*}
		where $\bar{I}$ is the complement of set $I$ and $|I|$ is the cardinality of $I$.
		Let us define two random variables
		\begin{equation*}
			\begin{split}
				A &=\sum_{j \in I} V_j = U\sum_{j \in I} X_j + (1-U)\sum_{j \in I}(1-X_j); \\
				B &=\sum_{j \in \bar{I}} V_j = U\sum_{j \in \bar{I}} X_j + (1-U)\sum_{j \in \bar{I}}(1-X_j).
			\end{split}
		\end{equation*}
		We first consider the case $d$ odd. By conditioning on the two possible outcome of the random variable $\XX$ we have:
		\begin{equation*}
			\begin{split}
				A|(\XX=\xx) &= |I|-k+(2k-|I|)U; \\
				B|(\XX=\xx) &= m_d-(|I|-k)-(1+2k-|I|)U; \\
				A|(\XX=\boldsymbol{1}_d-\xx) &= k-(2k-|I|)U; \\
				B|(\XX=\boldsymbol{1}_d-\xx) &= M_d-k+(1+2k-|I|)U.
			\end{split}
		\end{equation*}
		Let $(A_1,B_1)$ and $(A_2,B_2)$ be two independent copies of $(A,B)$. 
		We have, for $h=1,2$: 
		\begin{equation*}
			\begin{split}
				A_h & = U_h \sum_{j \in I} X^{(h)}_j + (1-U_h)\sum_{j \in I}(1-X^{(h)}_j); \\
				B_h & = U_h \sum_{j \in \bar{I}} X^{(h)}_j + (1-U_h)\sum_{j \in \bar{I}}(1-X^{(h)}_j),
			\end{split}
		\end{equation*}
		where $U_1$ and $U_2$ are two independent standard uniform, independent of $\XX^{(1)}$ and $\XX^{(2)}$ that are iid with $\XX$.
		We want to prove that $(A,B)$ is countermonotonic.
		\begin{equation} \label{prob_ctm}
			\begin{split}
				&\PP[(A_1-A_2)(B_1-B_2)\leq0]= \\
				&= \sum_{\xx_1}\sum_{\xx_2} \PP[(A_1-A_2)(B_1-B_2)\leq0|\XX^{(1)}=\xx_1, \XX^{(2)}=\xx_2]p^{(1)}(\xx_1)p^{(2)}(\xx_2),
			\end{split}
		\end{equation}
		where $p^{(h)}(\xx_h) = \PP[\XX^{(h)} = \xx_h]$.
		There are four possible values that the pair $(\xx_1,\xx_2)$ can assume: $(\xx,\xx), (\xx,\boldsymbol{1}_d-\xx), (\boldsymbol{1}_d-\xx,\xx)$ or $(\boldsymbol{1}_d-\xx,\boldsymbol{1}_d-\xx)$.
		
		Case 1. Let $(\xx_1,\xx_2) = (\xx,\xx)$.
		We have:
		\begin{equation} \label{prob_ctm_Case1}
			\begin{split}
				&\PP[(A_1-A_2)(B_1-B_2)\leq0|\XX^{(1)}=\xx, \XX^{(2)}=\xx] =\\
				&= \PP[(A_1|\xx-A_2|\xx)(B_1|\xx-B_2|\xx)\leq0] = \\
				&=\PP[(2k-|I|)(U_1-U_2)\cdot(-1)(1+2k-|I|)(U_1-U_2)\leq0] = \\
				&=\PP[-(2k-|I|)(1+2k-|I|)(U_1-U_2)^2\leq0] = 1,
			\end{split}
		\end{equation}
		where the last equality follows because $2k-|I| \in \mathbb{Z}$ and, if $2k-|I|\ge0$ then $1+2k-|I|>0$, while if $2k-|I|<0$ then $1+2k-|I|\le0$. 
		
		Case 2. Let $(\xx_1,\xx_2) = (\xx,\boldsymbol{1}_d-\xx)$.
		We have:
		\begin{equation*} \label{prob_ctm_Case2}
			\begin{split}
				&\PP[(A_1-A_2)(B_1-B_2)\leq0|\XX^{(1)}=\xx, \XX^{(2)}=\boldsymbol{1}_d-\xx] =\\
				&= \PP[(A_1|\xx-A_2|\boldsymbol{1}_d-\xx)(B_1|\xx-B_2|\boldsymbol{1}_d-\xx)\leq0] = \\
				&=\PP[(2k-|I|)(-1+U_1+U_2)\cdot(-1)(1+2k-|I|)(-1+U_1+U_2)\leq0] = \\
				&=\PP[-(2k-|I|)(1+2k-|I|)(U_1+U_2-1)^2\leq0] = 1,
			\end{split}
		\end{equation*}
		for the same argument in Case 1.
		
		Case 3 and Case 4 are analogous.  
		The case $d$ even is similar and the computations are simpler since $M_d=m_d$. 
		Finally, by~\eqref{prob_ctm}:
		\begin{equation*}
			\PP[(A_1-A_2)(B_1-B_2)\leq0] = \sum_{\xx_1} \sum_{\xx_2} 1 \cdot p^{(1)}(\xx_1)p^{(2)}(\xx_2) = 1.
		\end{equation*}
		Therefore, $\sum_{j\in I} V_j$ and $\sum_{j\in \bar{I}} V_j$ are countermonotonic, and $\VV$ is $\Sigma$-countermonotonic.
	\end{proof}
	
	The following Example~\ref{ex:CounterEx-CTM} shows that Proposition~\ref{prop:Sigma_ctm} holds only for extremal copulas and not for extremal mixture copulas.
	
	\begin{example} \label{ex:CounterEx-CTM}
		Let $d=5$. 
		Consider $\XX$ such that $\PP(\XX = \xx_1) = \PP(\XX = \xx_2) = \PP(\XX = \boldsymbol{1}_5-\xx_1) = \PP(\XX = \boldsymbol{1}_5-\xx_2) = 1/4$, with $\xx_1 = (1,0,0,0,1)$ and {$\xx_2 = (0,0,0,1,1)$}. Since $\XX$ has support in $\chi^*_5$, by Proposition~\ref{prop:supportBernoulliSumCX}, $\XX$ is $\Sigma_{cx}$-smallest.
		Let $I=\{1,5\}$, with $|I|=2$, and let $k_1 = \sum_{j\in I} x_{1,j} = 2$ and $k_2 = \sum_{j\in I} x_{2,j} = 1$.
		Therefore, it holds:
		\begin{align*}
			A|\xx_1 &= 2U; & B|\xx_1 &= 3-3U; & A|\boldsymbol{1}_5-\xx_1 &= 2-2U; & B|\boldsymbol{1}_5-\xx_1 &= 3U; \\
			A|\xx_2 &= 1; & B|\xx_2 &= 2-U; & A|\boldsymbol{1}_5-\xx_2 &= 1; & B|\boldsymbol{1}_2-\xx_1 &= 1+U.
		\end{align*}
		From~\eqref{prob_ctm}, conditioning on the events $\{\XX^{(1)}=\xx_1\}$ and $\{\XX^{(2)}=\xx_2\}$, we have:
		\begin{equation*}
			\PP[(A_1|\xx_1 - A_2|\xx_2)(B_1|\xx_1 - B_2|\xx_2) \leq 0] = \PP[(2U_1-1)(1-3U_1+U_2) \leq0] <1.
		\end{equation*}
		Therefore, $\VV = U\XX + (1-U)(1-\XX)$ is not $\Sigma$-countermonotonic.
	\end{example}
	
	\begin{proposition} \label{prop:V_joint_mix}
		Let $d$ be even. 
		If $\XX^{\ast} \in \BP$ is a $\Sigma_{cx}$-smallest element of its Fréchet class $\mathcal{SB}_d$, then $\VV^{\ast} = U \XX^{\ast} + (1-U)(1-\XX^{\ast})$ is a joint mix.
	\end{proposition}
	\begin{proof}
		$\XX^{\ast}$ is $\Sigma_{cx}$-smallest, therefore $\PP(\sum_{j=1}^d X_j = d/2) = 1$, since $d$ is even.
		We have
		\begin{equation*}
			\begin{split}
				\PP \bigg( \sum_{j=1}^d V_j^{\ast} = \frac{d}{2} \bigg) &= \PP \bigg( d(1-U) + (2U-1) \sum_{j=1}^d X_j^{\ast} = \frac{d}{2}  \bigg) 
				= \int_{0}^{1} \PP \bigg( \sum_{j=1}^d X_j^{\ast} = \frac{\frac{d}{2} - d(1-u)}{2u-1} \bigg) du  \\
				& = \int_{0}^{1} \PP \bigg( \sum_{j=1}^d X_j^{\ast} = \frac{d}{2}\bigg) du = 1.
			\end{split}
		\end{equation*}
	\end{proof}
	
	We conclude this section with an example of $\Sigma_{cx}$-smallest (thus $\Sigma$-countermonotonic) copula and an example of $\Sigma$-countermonotonic, but not $\Sigma_{cx}$-smallest, copula.
	
	\begin{example}
		Let $d=100$.
		Let $\VV^{(1)}$ be a random vector with uniform marginals corresponding to the copula
		\begin{equation}\label{eq:copula100}
			C_1(u_1,\dots,u_{100})=(\min_{j \leq 50}u_j + \min_{j \geq 51}u_j-1)^+.
		\end{equation}
		The copula $C_1$ in~\eqref{eq:copula100} corresponds to the symmetric $\Sigma_{cx}$-smallest Bernoulli distribution with support on $\xx^{(1)}= (\boldsymbol{1}_{50} // \boldsymbol{0}_{50})$, where $\boldsymbol{0}_{50}$ is a vector of length $50$ with all zeros, and $\boldsymbol{1}_{100}-\xx^{(1)}$. 
		By Proposition~\ref{prop:V_joint_mix}, $\VV^{(1)}$ is a joint-mix, therefore it is $\Sigma$-countermonotonic and a $\Sigma_{cx}$-smallest element in its Fréchet class.
		
		Let $d=103$. 
		Let $\VV^{(2)}$ be a random vector with uniform marginals corresponding to the copula
		\begin{equation}\label{eq:copula101}
			C_2(u_1,\dots,u_{103})=(\min_{j \leq 51}u_j + \min_{j \geq 52}u_j-1)^+.
		\end{equation}
		The copula $C_2$ in~\eqref{eq:copula101} corresponds to the symmetric $\Sigma_{cx}$-smallest Bernoulli distribution with support on $\xx^{(2)} = (\boldsymbol{1}_{51} // \boldsymbol{0}_{52})$ and $\boldsymbol{1}_{103}-\xx^{(2)}$. 
		According to Proposition~\ref{prop:Sigma_ctm}, the random vector $\VV^{(2)}$ is $\Sigma$-countermonotonic and a $\Sigma_{cx}$-smallest element in $\CCC_{103}^{\text{EM}}$.
		However, $\VV^{(2)}$ is not a $\Sigma_{cx}$-smallest element within its Fréchet class.
		In fact, we can see that $\VV^{(2)}$ is not a joint mix, as the sum of its components varies in the interval $(51,52)$; yet, its Fréchet class admits a joint mix.
		Consider, for example, $\VV^{(JM)} := (\VV^{(1)} // \VV^{(3)})$, where $\VV^{(3)}$ is a $3$-dimensional joint mix with uniform marginals, independent of $\VV^{(1)}$, and with the dependence structure specified in Example 3 in \cite{gaffke1981class}, .
		$\VV^{(1)}$ and $\VV^{(3)}$ are joint mixes, thus $\VV^{(JM)}$ is a joint mix and
		\begin{equation*}
			\sum_{j=1}^d V^{(JM)}_j \leq_{cx} \sum_{j=1}^d V^{(2)}_j.
		\end{equation*}
	\end{example}

	\subsubsection{FGM copulas} \label{sec:FGMcopulas}
	
	Another class of copulas that can be built from Bernoulli random vectors via a stochastic representation is the class $\CCC_d^{\text{FGM}}$ of multivariate Farlie-Gumbel-Morgenstern (FGM) copulas.
	In this section, we recall their definition and the stochastic representation introduced in \cite{blier2022stochastic}, that provides a one to one correspondence with the class $\BS$.
	In the class of FGM copulas the elements corresponding to $\Sigma_{cx}$-smallest Bernoulli random vectors are $\Sigma_{cx}$-smallest in the class $\CCC_d^{\text{FGM}}$, but not in the whole class of copulas. 
	It is therefore interesting to compare the minimal negative dependence in the class of FGM copulas, with the minimal dependence in the whole class of copulas. 
	This will be the focus of the next Section~\ref{sec:DependenceMeasures}. Here we presents some  results on FGM copulas that are not new, but necessary for the discussion in Section~\ref{sec:DependenceMeasures}.

	\begin{definition}
		A multivariate copula $C$ belongs to the class of FGM copulas if it has the following expression:
		\begin{equation*}
			C(\uu) = u_1 \cdots u_d  
			\left( 
			1+ \sum_{k=2}^d \sum_{1 \le j_1 < \ldots < j_k \le d} \theta_{j_1 \dots j_k} \bar{u}_{j_1}\cdots\bar{u}_{j_k} 
			\right), 
			\quad 
			\uu \in [0,1]^d,
		\end{equation*}
		where $\bar{u}_j = 1-u_j$, $j=1,\ldots,d$.
	\end{definition}
	There exist $2^d$ constraints on the parameters for the existence of a FGM copula (see \cite{cambanis1977some}), that are:
	\begin{equation*}
		1 + \sum_{k=2}^d \sum_{1 \le j_1 < \ldots < j_k \le d} \theta_{j_1 \dots j_k} \epsilon_{j_1}\epsilon_{j_2}\cdots \epsilon_{j_k} \ge 0, 
	\end{equation*}
	for every $(\epsilon_1,\ldots,\epsilon_d) \in \{-1,1\}^d$. When $d=2$, the admissible set for the unique parameter is the interval $[-1,1]$. 
	However, as the dimension increases, the shape of the set of admissible parameters becomes more and more complex.
	
	A useful tool to address this problem is the stochastic representation provided in \cite{blier2022stochastic}. 
	Let $\ZZ_0 = (Z_{1,0},\ldots,Z_{d,0})$ be a vector of independent exponential random variables with mean $\tfrac{1}{2}$ and let $\ZZ_1 = (Z_{1,1},\ldots,Z_{d,1})$ be a vector of independent exponential random variables with mean 1.
	Let $\ZZ_0$ and $\ZZ_1$ be independent.
	The following Theorem~\ref{FGMstoch_repr_THM}, proved in \cite{blier2022stochastic}, shows the existence of a one to one correspondence between the class $\BS$ and the class $\CCC_d^{\text{FGM}}$.
	
	\begin{theorem}\label{FGMstoch_repr_THM}
		Let $\XX \in \BS$ be a $d$-dimensional symmetric Bernoulli random vector. Let $\UU = (U_1,\ldots,U_d)$ be a random vector such that 
		\begin{equation}\label{FGMexpU}
			U_j = 1-\exp\{-(Z_{j,0} + X_j Z_{j,1})\}, \quad j \in {1,\ldots,d}.
		\end{equation}
		Then, $\UU$ has a $d$-variate distribution with standard uniform marginals and its cdf is a FGM copula given by
		\begin{equation*}\label{FGMcopula}
			C(\uu) = \sum_{\xx \in \mathcal{X}_d} f_{\XX}(\xx) \prod_{h=1}^d u_h \bigg( 1+ (-1)^{x_h} (1-u_h) \bigg), \quad \uu \in [0,1]^d.
		\end{equation*}
	\end{theorem}
	In \cite{blier2022stochastic}, the authors derive the parameters of the FGM copula in terms of the centered moments of its corresponding symmetric Bernoulli distribution:
	\begin{equation*}
		\theta_{j_1 \dots j_k} = (-2)^k \, \mathbb{E}_{\XX} \bigg[ \prod_{l=1}^k \bigg(X_{j_l}-\frac{1}{2} \bigg) \bigg]. 
	\end{equation*}
	
	Regarding extremal negative dependence in $\CCC_d^{\text{FGM}}$, in \cite{blier2022exchangeable}, the authors explicitly find the FGM copula that corresponds to the $\Sigma_{cx}$-smallest exchangeable Bernoulli distribution.
	Moreover, in a slightly more general context, the authors of \cite{cossette2024generalized} show that the one to one map between the classes $\BS$ and $\CCC_d^{\text{FGM}}$ preserves the convex order of the sums of the components.
	We restate Theorem 4.2 in \cite{cossette2024generalized} using the notation adopted in this paper.
	\begin{theorem}\label{cxFGM}
		Let $\XX, \XX' \in \BS$ and let $\UU$ and $\UU'$ be the corresponding uniform random vector with FGM copula.
		Then,
		\begin{equation*}
			\sum_{j=1}^d X_j \leq_{cx} \sum_{j=1}^d X'_j 
			\implies 
			\sum_{j=1}^d U_j \leq_{cx} \sum_{j=1}^d U'_j.
		\end{equation*}
	\end{theorem}
	
	Theorem~\ref{cxFGM} implies that FGM copulas corresponding to the $\Sigma_{cx}$-smallest element of $\BS$ are $\Sigma_{cx}$-smallest in the class of FGM copulas.  
	Using the characterization of all the $\Sigma_{cx}$-smallest element of $\BS$ in Section~\ref{sec:BernoulliEtremalDep}, we can investigate some properties of a $\Sigma_{cx}$-smallest FGM copula.
	
	\begin{remark} \label{rmk:palindromicFGM}
		We have already seen in Example~\ref{X_d=3} and Example~\ref{X_d=4} that the $\Sigma_{cx}$-smallest pmfs are palindromic in dimension $d \le 4$. By Proposition 3.4 in \cite{blier2022stochastic}, we know that the FGM copulas corresponding to the class $\BP$ are radially symmetric and, in particular, have $\theta_{j_1 \dots j_k} = 0$ for $1 \le j_1 < \ldots < j_k \le d$, for every odd value of $k$.
	\end{remark}

	\section{Minimal pairwise dependence measures} \label{sec:DependenceMeasures}
	
	The most common  dependence measures are Pearson's correlation $\rho_P$, Spearman's rho $\rho_S$, and Kendall's tau $\tau_K$, defined by
	\begin{equation*}
		\begin{split}
			\rho_P(X,Y)& = \frac{E[XY]-E[X]E[Y]}{\sqrt{\text{Var}(X) \text{Var}(Y)}}; \\
			\rho_S(X,Y)& = \rho_P(F_X(X), F_Y(Y)); \\
			\tau_K(X,Y)& = \PP((X-X')(Y-Y') \ge 0) - \PP((X-X')(Y-Y') \le 0).
		\end{split}
	\end{equation*}
	When the marginals are continuous, Spearman's rho and Kendall's tau are measures that depend only on the copula of the two random variables and not on their marginals (see \cite{nelsen2006introduction}).
	This is not true when the marginals are discrete, see Section 4.2 in \cite{genest2007primer}.
	It is evident that  $\rho_P(U,V) = \rho_S(U,V)$ for uniform random variables $U$ and $V$. 
	It is also easy to verify that Spearman's rho and Pearson's correlation are equal for Bernoulli random variables; therefore, we consider only the Pearson's correlation $\rho_P$.

	In this section, we  compare these measures for the $\Sigma$-countermonotonic vectors in $\BS$, the $\Sigma$-countermonotonic copulas in the class of extremal mixture copulas and the $\Sigma_{cx}$-smallest elements in the class of FGM copulas.
	
	The following proposition follows from direct computations.
	\begin{proposition} \label{prop:DependenceBernoulli}
		Let $X_i$ and $X_j$ be two Bernoulli random variables with mean $p_i$ and $p_j$, respectively. 
		Then, we have
		\begin{equation*}
			\rho_P(X_i,X_j) =  \frac{\tau_K(X_i,X_j)}{2\sqrt{p_i(1-p_i)p_j(1-p_j)}}.
		\end{equation*}
	\end{proposition}

	In \cite{fontana2018representation}, the authors prove that the upper and lower bounds for correlation are reached on the extremal points, allowing to find the range for possible correlations. From Proposition~\ref{prop:DependenceBernoulli} is follows that also Kendall's tau reaches its bounds on the extremal points.

	Since we also deal with dimensions $d>2$, we consider a simple approach to generalize these dependence measures to random vectors of dimension higher that two, that consists in averaging all pairwise measures (see the discussion in \cite{gijbels2021specification}).
	In particular, we denote the mean Pearson's correlation and the mean Kendall's tau of a $d$-dimensional random vector $\YY=(Y_1,\ldots,Y_d)$ as follows:
	\begin{equation*}
		\begin{split}
			\bar{\rho}_P(\YY) &= \frac{2}{d(d-1)} \sum_{1 \le i < j \le d} \rho_P(Y_i,Y_j), \\
			\\
			\bar{\tau}_K(\YY) &= \frac{2}{d(d-1)} \sum_{1 \le i < j \le d} \tau_K(Y_i,Y_j).
		\end{split}
	\end{equation*}
	\begin{proposition}\label{prop:measures}
		Let $\XX\in \BS$ and let $\VV \in \CCC_d^{\text{EM}}$ and $\UU \in \CCC_d^{\text{FGM}}$ be the corresponding uniform random vectors with extremal mixture copula and FGM copula, respectively. 
		We have 
		\begin{equation*}
			\begin{split}
				&\rho_P(X_{j_1},X_{j_2}) = 
				\rho_P(V_{j_1},V_{j_2}) = 3\rho_P(U_{j_1},U_{j_2}) \\
				&= 
				2 {\tau}_K(X_{j_1},X_{j_2}) = 
				2{\tau}_K(V_{j_1},V_{j_2}) =\frac{9}{2} \tau_K(U_{j_1},U_{j_2}),
			\end{split}
		\end{equation*}
		for every $j_1, j_2 \in \{1,\dots,d\}$, $j_1 \neq j_2$, and
		\begin{equation*}
			\bar{\rho}_P(\XX) = 
			\bar{\rho}_P(\VV) =3\bar{\rho}_P(\UU) = 
			2\bar{\tau}_K(\XX) =
			2\bar{\tau}_K(\VV) =\frac{9}{2} \bar{\tau}_K(\UU).
		\end{equation*}
		
	\end{proposition}
	\begin{proof}
		Let $j_1,j_2 \in \{1,\dots,d\}$, $j_1 \neq j_2$.
		From Proposition~\ref{prop:DependenceBernoulli}, we have that $\rho_P(X_{j_1},X_{j_2}) = 2\tau_K(X_{j_1},X_{j_2})$.
		Moreover, from the stochastic representation in~\eqref{eq:EMdistribution}, standard computations give $\rho_P(V_{j_1},V_{j_2}) = \rho_P(X_{j_1},X_{j_2})$ and $\tau_K(V_{j_1},V_{j_2}) = \tau_K(X_{j_1},X_{j_2})$.
		Regarding the FGM copula, from Point 4 of Corollary 3.1 in \cite{cossette2024generalized}, it follows that $\rho_P(X_{j_1},X_{j_2}) = 3\rho_P(U_{j_1},U_{j_2})$.
		Since $\UU$ has distribution in $\CCC_d^{\text{FGM}}$, the cdf of the pair $(U_{j_1},U_{j_2})$ is a bivariate FGM copula.
		It is well known that $\rho_P(U_{j_1},U_{j_2}) = \frac{\theta}{3}$ and $\tau_K(U_{j_1},U_{j_2}) = \frac{2\theta}{9}$, where $\theta \in [-1,1]$ is the unique parameter of the bivariate FGM copula.
		Therefore, $\rho_P(U_{j_1},U_{j_2}) = \frac{3}{2} \tau_K(U_{j_1},U_{j_2})$ and the thesis follows.
	\end{proof}

	As a consequence of Proposition~\ref{prop:measures}, the analysis of the pairwise dependence measures of extremal mixture copulas and FGM copulas is fully described by the pairwise dependence measures of the corresponding symmetric Bernoulli distributions.
	Moreover, Corollary~\ref{cor:PalindromicAreSufficient} shows that it is sufficient to consider palindromic Bernoulli distributions, and the corresponding copulas, to describe all the possible structures of pairwise dependence measures.
	Indeed, the authors of \cite{huber2019admissible} proved that for every $\XX \in \BS$ there exists $\XX' \in \BP$ with the same bivariate Pearson's correlation structure.
	\begin{corollary} \label{cor:PalindromicAreSufficient}
		Let $\XX \in \BS$ and let $\UU$ be a uniform random vector with the FGM copula corresponding to $\XX$.
		Then there exists $\XX' \in \BP$ such that $\rho_P(U_{j_1},U_{j_2}) = \rho_P(U'_{j_1},U'_{j_2})$ and $\tau_K(U_{j_1},U_{j_2}) = \tau_K(U'_{j_1},U'_{j_2})$, for every $j_1,j_2 \in \{1,\dots,d\}$, $j_1 \neq j_2$, where $\UU'$ is a uniform random vector with FGM copula corresponding to $\XX'$. 
		Moreover, the extremal mixture copulas corresponding to $\XX$ and $\XX'$ via the stochastic representation in~\eqref{eq:EMdistribution} coincide.
	\end{corollary}
	\begin{proof}
		Given $\XX \in \BS$, from Theorem 1 in \cite{huber2019admissible}, there exists $\XX' \in \BP$ such that $\rho_P(X_{j_1},X_{j_2}) = \rho_P(X'_{j_1},X'_{j_2})$, for every $j_1,j_2 \in \{1,\dots,d\}$, $j_1 \neq j_2$.
		Moreover, by Proposition~\ref{prop:measures}, it holds $\rho_P(U_{j_1},U_{j_2}) = \rho_P(U'_{j_1},U'_{j_2})$ and $\tau_K(U_{j_1},U_{j_2}) = \tau_K(U'_{j_1},U'_{j_2})$, for every $j_1,j_2 \in \{1,\dots,d\}$, $j_1 \neq j_2$.
		From the construction of $\XX' \in \BP$ in the proof of Theorem 1 in \cite{huber2019admissible}, $\XX$ and $\XX'$ are such that
		\begin{equation*}
			f(\ss_{\ii}) + f(\boldsymbol{1}_d - \ss_{\ii})
			=
			f'(\ss_{\ii}) + f'(\boldsymbol{1}_d - \ss_{\ii}),
		\end{equation*}
		for every $\ii \in \mathcal{X}_{d-1}$.
		Therefore, the weights $w_{\ii}$, given in~\eqref{eq:EMweights}, of the corresponding extremal mixture copulas are equal and these copulas coincide.
	\end{proof}
	
	The following Corollary~\ref{cor:minimalMeasures} to Proposition~\ref{prop:measures} states that extremal mixture copulas and FGM copulas built from $\Sigma_{cx}$-smallest Bernoulli random vectors have minimal mean correlation and Kendall's tau. 
	Its proof relies on the following known fact: the mean Pearson's correlation of a Bernoulli random vector $\XX \in \BS$ can be expressed as the expectation of a convex function of the sum $S_{\XX} = X_1+\dots+X_d$, i.e.
	\begin{equation} \label{eq:MinimalMeanCorrelation}
		\bar{\rho}_P(\XX) = E[\phi(S_{\XX})],
	\end{equation}
	where
	\begin{equation} \label{eq:phi}
		\phi(y) = 
		\begin{cases}
			\frac{8}{d(d-1)}\binom{y}{2} - 1, &\text{if } y \geq 2
			\\
			0, &\text{otherwise}
		\end{cases}.
	\end{equation}

	In particular, if $\XX$ is a $\Sigma_{cx}$-smallest element in $\BS$, since $\phi$ is a convex function, $\bar{\rho}_P(\XX)$ is minimal and we have
	\begin{equation} \label{eq:ExplicitMinimalMeanCorrelation}
		\bar{\rho}_P(\XX) = 
		\begin{cases}
			- \frac{1}{d-1}, &\text{if $d$ is even}
			\\
			- \frac{1}{d}, &\text{if $d$ is odd}
		\end{cases}.
	\end{equation}
	We notice that, if $d$ is odd, the minimal mean Pearson's correlation in the class $\BS$ is equal to the minimal mean Pearson's correlation in the class $\mathcal{SB}_{d+1}$.

	\begin{corollary} \label{cor:minimalMeasures}
		Let $\XX\in \BS$ be a $\Sigma_{cx}$-smallest element in $\BS$ and let $\VV \in \CCC_d^{\text{EM}}$ and $\UU \in \CCC_d^{\text{FGM}}$ be the corresponding uniform random vectors with extremal mixture copula and FGM copula, respectively. We have 
		\begin{equation*}
			\begin{split}
				\bar{\rho}_P(\VV) \le  \bar{\rho}_P(\VV')
				\quad \text{and} \quad
				\bar{\tau}_K(\VV) \le  \bar{\tau}_K(\VV'),
			\end{split}
		\end{equation*}
		for any $\VV '\in \mathcal{C}_d^{EM}$, and 
		\begin{equation*}
			\begin{split}
				\bar{\rho}_P(\UU) \le  \bar{\rho}_P(\UU')
				\quad \text{and} \quad
				\bar{\tau}_K(\UU) \le  \bar{\tau}_K(\UU'),
			\end{split}
		\end{equation*}
		for any $\UU' \in \mathcal{C}_d^{\text{FGM}}$.
		Moreover,
		\begin{equation*}
			\bar{\rho}_P(\VV) =
			\begin{cases}
				- \frac{1}{d-1}, &\text{if $d$ is even}
				\\
				- \frac{1}{d}, &\text{if $d$ is odd}
			\end{cases},
			\quad
			\bar{\tau}_K(\VV) = 
			\begin{cases}
				- \frac{1}{2(d-1)}, &\text{if $d$ is even}
				\\
				- \frac{1}{2d}, &\text{if $d$ is odd}
			\end{cases},
		\end{equation*}
		and
		\begin{equation*}
			\bar{\rho}_P(\UU) =
			\begin{cases}
				- \frac{1}{3(d-1)}, &\text{if $d$ is even}
				\\
				- \frac{1}{3d}, &\text{if $d$ is odd}
			\end{cases},
			\quad
			\bar{\tau}_K(\UU) = 
			\begin{cases}
				- \frac{2}{9(d-1)}, &\text{if $d$ is even}
				\\
				- \frac{2}{9d}, &\text{if $d$ is odd}
			\end{cases}.
		\end{equation*}
	\end{corollary}
	\begin{proof}
		Let $\XX \in \BS$ be a $\Sigma_{cx}$-smallest element in $\BS$.
		From~\eqref{eq:MinimalMeanCorrelation} and~\eqref{eq:phi}, since $\phi$ in~\eqref{eq:phi} is a convex function, $\bar{\rho}_P(\XX) \leq  \bar{\rho}_P(\XX')$, for any $\XX' \in \BS$.
		Since every extremal mixture copula and every FGM copula can be build from a Bernoulli random vector $\XX' \in \BS$, the thesis follows from Proposition~\ref{prop:measures} and from~\eqref{eq:ExplicitMinimalMeanCorrelation}.
	\end{proof}
	
	Every $\Sigma_{cx}$-smallest Bernoulli random vector, every extremal mixture copula, and every FGM copula built from a $\Sigma_{cx}$-smallest Bernoulli random vector have the same  mean Pearson's correlation that depends only on the dimension of the class.
	However, $\Sigma_{cx}$-smallest Bernoulli random vectors in the same class $\BS$ have different dependence structures.
	We now study pairwise dependence measures corresponding to different $\Sigma_{cx}$-smallest Bernoulli random vectors in the class $\BS$.
	We first consider a $\Sigma_{cx}$-smallest Bernoulli random vector $\XX^{K}$ with pmf $f^K \in \mathcal{B}_K$. 
	We know that there exists $\xx \in \mathcal{X}^*_{d}$ such that $f^K(\xx)=f^K(\boldsymbol{1}_d - \xx) = \tfrac{1}{2}$.
	It follows that there are $n^+_d$ comonotonic pairs (a pair is comonotonic if one variable is a deterministic increasing transformation of the other, see \cite{puccetti2015extremal}), where
	\begin{equation*}
		n^+_d =
		\begin{cases}
			\frac{d(d-2)}{4}, &\text{if $d$ is even}
			\\
			\frac{(d-1)^2}{4}, &\text{if $d$ is odd}
		\end{cases},
	\end{equation*}
	and $n^-_d$ countermonotonic pairs, where
	\begin{equation*}
		n^-_d = \binom{d}{2} - n^+_d = 
		\begin{cases}
			\frac{d^2}{4}, &\text{if $d$ is even}
			\\
			\frac{(d-1)(d+1)}{4}, &\text{if $d$ is odd}
		\end{cases}.
	\end{equation*}
	
	Obviously, countermonotonic pairs have correlation $\rho_P=-1$ and comonotonic pairs have correlation $\rho_P=1$, and the mean Pearson's correlation of a random vector $\XX^K \in \mathcal{B}_K$ is given by $\frac{2}{d(d-1)}(n_d^+ - n_d^-)$, that is equal to~\eqref{eq:ExplicitMinimalMeanCorrelation}.
	Since the extremal copulas are in a one to one relationship with the elements in $\mathcal{B}_K$, and $\rho_P(X_{j_1},X_{j_2}) = \rho_P(V_{j_1},V_{j_2})$, we can conclude that extremal copulas built from $\Sigma_{cx}$-smallest Bernoulli random vectors have $n_d^+$ comonotonic pairs and $n_d^-$ countermonotonic pairs (and they are $\Sigma$-countermonotonic by Proposition~\ref{prop:Sigma_ctm}).
	
	We can then consider the unique exchangeable $\Sigma_{cx}$-smallest Bernoulli random vector $\XX^e \in \BS$.
	In this case, for every $j_1,j_2 \in \{1,\dots,d\}$, $j_1 \neq j_2$, we have $\rho_P(X^e_{j_1},X^e_{j_2}) = \bar{\rho}_P(\XX^e)$, where $\bar{\rho}_P(\XX^e)$ is given by~\eqref{eq:ExplicitMinimalMeanCorrelation}.
	In this case, $\XX^e$ and the uniform random vectors $\VV^e$ and $\UU^e$, respectively with the extremal mixture copula and FGM copula corresponding to $\XX^e$, are pairwise negatively correlated (all the pairwise Pearson's correlations are non-positive).
	See \cite{cossette2025extremal}, for a more detailed analysis of the properties of $\XX^e$.

	Corollary~\ref{cor:PalindromicAreSufficient} implies that $\rho_P$ and $\tau_K$ of a vector $\UU$ with FGM copula are uniquely determined from $\rho_P$ and $\tau_K$ of a properly chosen Palindromic Bernoulli random vector.
	We conclude this section by showing that this does not hold true when other dependence measures are considered.
	Let $\XX$ be a Bernoulli random vector with pmf $f \in \BS$, but $f \notin \BP$, and let $\VV$ and $\UU$ be the uniform random vectors with the extremal mixture copula and FGM copula built from $\XX$, respectively.
	By Corollary~\ref{cor:PalindromicAreSufficient}, there exists $\XX'$ with pmf $f' \in \BP$ such that $\VV$ has the same distribution of $ \VV'$ and $\rho_P(U_{j_1},U_{j_2}) = \rho_P(U'_{j_1},U'_{j_2})$, for every $j_1,j_2 \in \{1,\dots,d\}$, $j_1 \neq j_2$, where $\VV'$ and $\UU'$ are the uniform random vectors with the extremal mixture copula and FGM copula corresponding to $\XX'$, respectively.
	To study the differences in the dependence structures of $\XX$ and of $\XX'$, and of the corresponding FGM copulas, we define the centered cross moments of order three of a $d$-dimensional random vector $\YY=(Y_1,\ldots,Y_d)$ as
	\begin{equation*}
		\tilde{\mu}_{j_1,j_2,j_3}(\YY) 
		= 
		E \left[ \prod_{h=1}^3 \left( 
		\frac{Y_{j_h} - E[Y_{j_h}]}{\sqrt{\text{Var}(Y_{j_h})}} 
		\right) \right],
	\end{equation*}
	for $j_1,j_2,j_3 \in \{1,\dots,d\}$, $j_1 \neq j_2 \neq j_3$.
	
	\begin{proposition} \label{prop:PalindromicCrossMoments}
		Let $\XX' \in \BP$ and let $\VV'$ and $\UU'$ be the corresponding uniform random vectors with extremal mixture copula and FGM copula, respectively.
		Then, $\tilde{\mu}_{j_1,j_2,j_3}(\XX') = \tilde{\mu}_{j_1,j_2,j_3}(\VV') = \tilde{\mu}_{j_1,j_2,j_3}(\UU') = 0$, for every $j_1,j_2,j_3 \in \{1,\dots,d\}$, $j_1 \neq j_2 \neq j_3$.
	\end{proposition}
	\begin{proof}
		It is easy to show that if a random vector $\YY$ has the same distribution of $\boldsymbol{1}_d - \YY$, then $\tilde{\mu}_{j_1,j_2,j_3}(\YY) = 0$, for every $j_1,j_2,j_3 \in \{1,\dots,d\}$, $j_1 \neq j_2 \neq j_3$.
		A palindromic Bernoulli random vector is such that $\XX'$ has the same distribution of $ \boldsymbol{1}_d - \XX'$ and every extremal mixture copula is radially symmetric, i.e.\@ $\VV'$  has the same distribution of $ \boldsymbol{1}_d - \VV'$.
		Finally, by Proposition 3.4 in \cite{blier2022stochastic}, if $\XX' \in \BP$, the corresponding FGM copula is radially symmetric. 
	\end{proof}
	
	Standard computations give the following Proposition.
	\begin{proposition} \label{prop:CenteredCrossMoments3}
		Let $\XX \in \BS$ and let $\VV$ and $\UU$ be the corresponding uniform random vectors with extremal mixture copula and FGM copula, respectively.
		Then, $\tilde{\mu}_{j_1,j_2,j_3}(\XX) = \tilde{\mu}_{j_1,j_2,j_3}(\VV) = -\frac{\sqrt{3}}{9}\tilde{\mu}_{j_1,j_2,j_3}(\UU)$, for every $j_1,j_2,j_3 \in \{1,\dots,d\}$, $j_1 \neq j_2 \neq j_3$.
	\end{proposition}
	
	We conclude this section with an example in dimension $d=6$.
	
	\begin{example}
		Let us consider $\ff^{(1)}$ of Example~\ref{ex:X_d=6}. Clearly, $\ff^{(1)} \in \BS$, but $\ff^{(1)} \notin \BP$
		Let $\XX$ be a Bernoulli random vector with pmf $\ff^{(1)}$, and let $\VV$ and $\UU$ be the corresponding uniform random vectors with extremal mixture copula and FGM copula, respectively.
		From Corollary~\ref{cor:PalindromicAreSufficient}, there exists a Bernoulli random vector $\XX'$, with pmf $\ff' \in \BP$, such that $\rho_P(U_{j_1},U_{j_2}) = \rho_P(U'_{j_1},U'_{j_2})$, for every $j_1,j_2 \in \{1,\dots,d\}$, $j_1 \neq j_2$, where $\UU'$ is a uniform random vector with the FGM copula corresponding to $\XX'$.
		Moreover, $\VV$ has the same distribution of $\VV'$, where $\VV'$ is a uniform random vector with the extremal mixture copula corresponding to $\XX'$.
		However, it is not true that $\UU$ has the same distribution of $\UU'$. 
		Indeed, by Proposition~\ref{prop:PalindromicCrossMoments}, we have $\tilde{\mu}_{j_1,j_2,j_3}(\XX) = 0$ for every $j_1,j_2,j_3 \in \{1,\dots,d\}$, $j_1 \neq j_2 \neq j_3$, while 
		\begin{equation*}
			\tilde{\mu}_{j_1,j_2,j_3}(\XX) =
			\begin{cases}
				1, &\text{if } (j_1,j_2,j_3) \in \{(1,2,4),(1,3,5),(2,5,6),(3,4,6)\}
				\\
				-1, &\text{if } (j_1,j_2,j_3) \in \{(3,5,6),(2,4,6),(1,3,4),(1,2,5)\}
				\\
				0, &\text{otherwise}
			\end{cases}.
		\end{equation*}
		and, by Proposition~\ref{prop:CenteredCrossMoments3}, $\tilde{\mu}_{j_1,j_2,j_3}(\UU) =  -3\sqrt{3}\tilde{\mu}_{j_1,j_2,j_3}(\XX)$, for every $j_1,j_2,j_3 \in \{1,\dots,d\}$, $j_1 \neq j_2 \neq j_3$.
	\end{example}
	
	\section{Conclusion}\label{Concl}

	Some classes of copulas can be built using multivariate symmetric Bernoulli distributions, inheriting certain dependence properties.
	We study the minimal risk and extremal negative dependence distributions of multivariate symmetric Bernoulli distributions and  characterize the dependence properties of the corresponding copulas. In doing so, we also explicitly identify a class of $\Sigma$-countermonotonic copulas. 
	The connection between copulas and Bernoulli distributions has proven effective  in deriving statistical properties of families of copulas, such as minimal correlation. In this context, the recent article \cite{cossette2024generalized} investigates the characterization of extremal negative dependence within the class of FGM copulas and some of their generalizations.
	A key role in our findings is played by the geometric and algebraic structure of multivariate Bernoulli distributions, which has its own theoretical interest and warrants further investigation in our future research.

	\section*{Acknowledgements}
	Patrizia Semeraro gratefully acknowledges financial support from the INdAM-GNAMPA project $\text{CUP}\_ \text{E53C22001930001}$.
	
	\appendix

	\section{Proofs of Section~\ref{sec:SymmetricBernoulli}} \label{app:ProofFirstSection}
	
	\begin{proof}[Proof of Proposition~\ref{prop:extremal_pmfs_kernel}]
		The pmfs in $\mathcal{B}_K$ have mass on two points only. 
		Let us suppose that a pmf $\ff^K \in \mathcal{B}_K$ is not extremal. 
		Therefore, there exists an extremal pmf $\rr \in \BS$ whose support is contained in the support of $\ff^K$, as a consequence of Lemma 2.4 in \cite{terzer2009large}. 
		Therefore, the extremal pmf $\rr$ has support on one point only.
		However, this is not consistent with the condition on the marginal means equal to $\tfrac{1}{2}$.
	\end{proof}
	
	\begin{proof}[Proof of Proposition~\ref{prop:EquivalentPolynomials}]
		Let $P(\zz)$ and $Q(\zz)$ be two equivalent polynomials and let $\ff^P$ and $\ff^Q$ be the two corresponding type-0 pmfs. 
		By Definition~\ref{def:EquivalentPolynomial}, there exists $\mu >0$ such that $P(\zz) = \mu Q(\zz)$.
		Let us apply Algorithm 1 both to $P(\zz)$ and to $Q(\zz)$.  
		The coefficients of the two polynomials have the same sign, therefore, after the first step of the algorithm, we have $\ff^P = \mu \ff^Q$. 
		After normalization, we have that the two type-0 pmfs are identical.
	\end{proof}
	
	\begin{remark}\label{rmk:OtherMarginalMeans}
		It is worth noting that Proposition~\ref{prop:EquivalentPolynomials} holds for any Fréchet class of joint distributions with one dimensional Bernoulli random variables with common mean $p$, $p \in [0,1] \cap \mathbb{Q}$.
		For $p \neq \tfrac{1}{2}$, the argument is analogous, but the proof relies on Algorithm 1 in \cite{fontana2024high}. 
	\end{remark}
	
	\begin{proof}[Proof of Proposition~\ref{prop:inverseMap}]
		Let $A = \{ \ff \in \BS : \ff = \lambda \ff^P + (1-\lambda) \ff^K \text{, with } \ff^K \in \mathcal{K(H)}, \lambda \in (0,1] \}$, where $\ff^P$ is the type-0 pmf of $P(\zz)$  and, without any restiction, suppose that $P(\zz)$ is such that $\HHH(\ff^P) = P(\zz)$.
		The map $\HHH$ is linear, therefore $\HHH(\ff) = \HHH(\lambda \ff^P + (1-\lambda) \ff^K) = \lambda \HHH(\ff^P) + (1-\lambda) \HHH(\ff^K) = \lambda P(\zz)$, and we have $A \subseteq \HHH^{-1}[P(\zz)]$.
		Let $\ff \in \HHH^{-1}[P(\zz)]$. 
		We want to prove that $\ff$ can be written as a convex linear combination between $\ff^P$ and an element of the kernel of $\HHH$. 
		Let $Q(\zz)=\mathcal{H}(\ff)$.
		It follows that, by definition of $\HHH^{-1}$, there exists $\mu \in (0,1]$ such that $\HHH(\ff) = \mu \HHH(\ff^P)$, that is $Q(\zz) = \mu P(\zz)$. 
		Hence, to prove that $\ff = \mu \ff^P + (1-\mu)\ff^K$, we need to prove that $\ff^K = \frac{1}{1-\mu}\ff - \frac{\mu}{1-\mu} \ff^P$ is a pmf with null polynomial.
		We have that
		\begin{equation*}
			\HHH(\ff^K) = \HHH\bigg(\frac{1}{1-\mu}\ff - \frac{\mu}{1-\mu} \ff^P\bigg) = \frac{1}{1-\mu}\HHH(\ff) - \frac{\mu}{1-\mu} \HHH(\ff^P) = 0.
		\end{equation*}
		Also, since $\frac{1}{1-\mu} - \frac{\mu}{1-\mu} = 1$, the components of $\ff^K$ have sum equal to one. 
		We have to prove that the components of $\ff^K$ are non-negative. 
		For every $\ii \in \mathcal{X}_{d-1}$, we have $a_{\ii}^Q = \mu a_{\ii}^P$, where $a_{\ii}^Q$ and $a_{\ii}^P$ are the coefficients of the polynomials $Q(\zz)$ and $P(\zz)$, respectively.
		Then, from~\eqref{eq:mapH2}, $f(\ss_{\ii}) - f(\boldsymbol{1}_d - \ss_{\ii}) = \mu f^P(\ss_{\ii}) - \mu f^P(\boldsymbol{1}_d - \ss_{\ii})$.
		Thus, we have
		\begin{equation} \label{eq1}
			f(\ss_{\ii}) - \mu f^P(\ss_{\ii}) = f(\boldsymbol{1}_d - \ss_{\ii}) - \mu f^P(\boldsymbol{1}_d - \ss_{\ii}).
		\end{equation}
		By construction, we have that the type-0 pmf has either $f^P(\ss_{\ii}) = 0$ or $f^P(\boldsymbol{1}_d - \ss_{\ii}) = 0$. 
		In the first case,~\eqref{eq1} becomes $f(\boldsymbol{1}_d - \ss_{\ii}) - \mu f^P(\boldsymbol{1}_d - \ss_{\ii}) = f(\ss_{\ii}) \ge 0$, while in the other case,~\eqref{eq1} becomes $f(\ss_{\ii}) - \mu f^P(\ss_{\ii}) = f(\boldsymbol{1}_d - \ss_{\ii}) \ge 0$. 
		For each vector $\xx \in \design$, there exists $\ii \in \mathcal{X}_{d-1}$ such that $\xx = \ss_{\ii}$ or $\xx = \boldsymbol{1}_d - \ss_{\ii}$. 
		Thus, for every $\xx \in \design$, we have $f(\xx) - \mu f^P(\xx) \ge 0$.
		Therefore, the components of $\ff^K$ are always non-negative, since $f^K(\xx) = \frac{1}{1-\mu}(f(\xx)-\mu f^P(\xx)) \ge 0$, for all $\xx \in \design$.
		Hence, $\ff^K$ is a pmf of the kernel of $\HHH$ and we have $\HHH^{-1}[P(\zz)] \subseteq A$. 
		It follows $\HHH^{-1}[P(\zz)] \equiv A$.
	\end{proof}
	
	\begin{proof}[Proof of Proposition~\ref{prop:extipe0}]
		Let $\ff$ be an extremal pmf of $\BS$ and  let $\ff \in \HHH^{-1}[P(\zz)]$.
		By Proposition~\ref{prop:inverseMap}, there exist an element of the kernel of $\HHH$, $\ff^K \in \mathcal{K(H)}$, and $\lambda \in (0,1]$ such that
		\begin{equation*}
			\ff = \lambda \ff^P +(1-\lambda) \ff^K,
		\end{equation*}
		where $\ff^P$ is the type-0 pmf of $P(\zz)$. 
		Since $\ff$ is an extremal point, there does not exist any convex combination of elements of the polytope, that are different from $\ff$ and that generates $\ff$.
		Therefore, since $\HHH(\ff)= \HHH(\ff^P) =P(\zz)$, while $\HHH(\ff^K) \equiv 0$ we have $\ff = \ff^P$.
	\end{proof}

	\section{Proofs of Section~\ref{sec:minCX_ExtremalNegDep}} \label{app:ProofSecondSection}
	
	\begin{proof}[Proof of Proposition~\ref{prop:minCX_kernel}]
		Since $\ff^{K\ast}$ is an element of the kernel of $\HHH$, it is a convex linear combination of the basis $\mathcal{B}_K$, that is a set of extremal pmfs, given Proposition~\ref{prop:extremal_pmfs_kernel}.
		Furthermore, $\ff^{K\ast}$ is $\Sigma_{cx}$-smallest, therefore by Proposition~\ref{prop:supportBernoulliSumCX} $\ff^{K\ast}$, it has support on  $\mathcal{X}_d^{\ast}$. Hence, only the pmfs with support in $\mathcal{X}_d^{\ast}$ can have coefficients of the convex linear combination different from zero.
	\end{proof}
	
	\begin{proof}[Proof of Theorem~\ref{thm:polyminCX}]
		First, since $M_d + m_d = d$, we can notice that $\xx \in \mathcal{X}_d^{\ast}$ if and only if $\boldsymbol{1}_{d} - \xx \in \mathcal{X}_d^{\ast}$.
		Let $\ff$ be a pmf $\Sigma_{cx}$-smallest in $\BS$. 
		Then, $f(\xx) = 0$ for every $\xx \notin \mathcal{X}_d^{\ast}$. Let $\ii = (i_1,\ldots,i_{d-1}) \notin \I_{d-1}^{\ast}$. Clearly, $\ss_{\ii}=(i_1,\ldots,i_{d-1},0) \notin \mathcal{X}_{d}^{\ast}$ and $\boldsymbol{1}_d-\ss_{\ii}=(1-i_1,\ldots,1-i_{d-1},1) \notin \mathcal{X}_{d}^{\ast}$.
		It follows that $f(\ss_{\ii}) = f(\boldsymbol{1}_d-\ss_{\ii}) = 0$, and, by~\eqref{eq:mapH2}, $a_{\ii} = 0$ for every $\ii \notin \I_{d-1}^{\ast}$.
		
		The other two points of the theorem directly follow by Corollary 3.1 in \cite{fontana2024high}. 
		It states that all the polynomials of $\III_{\PPP}$ are linear combinations of the following polynomials, called fundamental polynomials:
		\begin{equation*}
			F_{\ii}(\zz) = F_{j_1,\ldots,j_{n_{\ii}}}(\zz) = \prod_{h=1}^{n_{\ii}} z_{j_h} - \sum_{h=1}^{n_{\ii}} z_{j_h} + (n_{\ii}-1) = \zz^{\ii} - \sum_{h=1}^{n_{\ii}} z_{j_h} + (n_{\ii}-1),
		\end{equation*}
		where $\ii \in \mathcal{X}_{d-1}$ is the vector with ones in the positions $(j_1,\ldots,j_{n_{\ii}})$ and zeros elsewhere and $n_{\ii} := \sum_{j=1}^{d-1}i_j \ge 2$.
		See \cite{fontana2024high} for further details.
		Let $P^*(z)$ be a  polynomial of a $\Sigma_{cx}$-smallest pmf. We can write $P^{\ast}(\zz)$ as a linear combination of the fundamental polynomials:
		\begin{equation}\label{Pstar}
			P^{\ast}(\zz) 
			=
			\sum_{\ii \in \mathcal{X}_{d-1}} \gamma_{\ii} F_{\ii}(\zz) = \sum_{\ii \in \mathcal{X}_{d-1}} \gamma_{\ii} \bigg( \zz^{\ii} - \sum_{h=1}^{n_{\ii}} z_{j_h} + (n_{\ii}-1) \bigg), 
		\end{equation}
		with $\gamma_{\ii} \in \RR$, for every $\ii \in \mathcal{X}_{d-1}$.
		From Point~\ref{null_coeff} we have
		\begin{equation}\label{poli_minCX_inThm}
			P^{\ast}(\zz) = \sum_{\ii \in \I_{d-1}^{\ast}} a_{\ii} \zz^{\ii}.
		\end{equation}
		Therefore, in Equation~\eqref{Pstar}, $\gamma_{\ii} = 0$, for every $\ii \notin \I^{\ast}_{d-1}$ and $\gamma_{\ii} = a_{\ii}$ for every $\ii \in \I^{\ast}_{d-1}$.
		In order to have a polynomial in the form of~\eqref{poli_minCX_inThm}, the linear and the constant terms must vanish. 
		We can write the linear terms as:
		\begin{equation*}
			-\sum_{\ii \in \I_{d-1}^{\ast}} a_{\ii} \sum_{h=1}^{n_{\ii}} z_{j_h} = -\sum_{j=1}^{d-1} z_j \sum_{\ii \in \I_{d-1}^{\ast} : i_j=1} a_{\ii}.
		\end{equation*}
		Since all the linear terms must vanish, we have $\sum_{\ii \in \I_{d-1}^{\ast} : i_j=1} a_{\ii} = 0$ that proves Point~\ref{sum_linear}.
		Finally, since $n_{\ii}$ is equal to $M_d$ or $m_d$, the constant term can be written as:
		\begin{equation} \label{nullsum1}
			\sum_{\ii \in \mathcal{X}_{d-1}} a_{\ii} (n_{\ii}-1) = (M_d-1) \sum_{\ii \in \I_{d-1}^{M_d}} a_{\ii} + (m_d-1) \sum_{\ii \in \I_{d-1}^{m_d}} a_{\ii},
		\end{equation}
		where $\I_{d-1}^k = \{ \ii \in \mathcal{X}_{d-1} : \sum_{h=1}^{d-1} i_h = k\}$. 
		We also know that the polynomials must vanish at the points $\PPP = \{ \boldsymbol{1}_{d-1}, \boldsymbol{1}_{d-1}^{-j}, j=1,\ldots,d-1 \}$, in particular, 
		\begin{equation} \label{nullsum2}
			P^{\ast}(\boldsymbol{1}_{d-1}) = \sum_{\ii \in \I^{\ast}_{d-1}} a_{\ii} = \sum_{\ii \in \I_{d-1}^{M_d}} a_{\ii} + \sum_{\ii \in \I_{d-1}^{m_d}} a_{\ii} = 0.
		\end{equation}
		From~\eqref{nullsum1} and~\eqref{nullsum2} we have
		\begin{equation*}
			\sum_{\ii \in \I_{d-1}^{M_d}} a_{\ii} = \sum_{\ii \in \I_{d-1}^{m_d}} a_{\ii} = 0,
		\end{equation*}
		that proves Point~\ref{sum_constants}.
	\end{proof}
	
	\begin{proof}[Proof of Corollary~\ref{corollary_system}]
		Point~\ref{sum_constants} of Theorem~\ref{thm:polyminCX} implies 
		\begin{equation}\label{point2}
			\begin{cases}
				\sum_{\ii \in \I_{d-1}^{M_d}} a_{\ii} =0\\
				\sum_{\ii \in \I_{d-1}^{m_d}} a_{\ii} = 0
			\end{cases},
		\end{equation}
		where $\I_{d-1}^k = \{ \ii \in \mathcal{X}_{d-1} : \sum_{h=1}^{d-1} i_h = k\}$.
		From  Point~\ref{sum_linear} of Theorem~\ref{thm:polyminCX} we have that for every $j \in \{1,\ldots,d-1\}$:
		\begin{equation}\label{point3}
			\sum_{\ii \in \I_{d-1}^{\ast} : i_j=1} a_{\ii} =0.  
		\end{equation}
		Since $\sum_{\ii \in \I_{d-1}^{\ast} : i_j=1} a_{\ii} =\sum_{\ii \in \I_{d-1}^{\ast}} i_j a_{\ii} $,~\eqref{point3} becomes 
		\begin{equation*}
			\sum_{\ii \in \I_{d-1}^{\ast}} i_j a_{\ii}=0,  {\quad j = 1,\ldots,d-1}.
		\end{equation*}
		We consider two cases.
		Case 1: $d$ odd.~\eqref{point2} and~\eqref{point3} lead to the linear system
		\begin{equation*}
			\begin{cases}
				\sum_{\ii \in \I_{d-1}^{M_d}} a_{\ii} =0\\
				\sum_{\ii \in \I_{d-1}^{m_d}} a_{\ii} = 0\\
				\sum_{\ii \in \I_{d-1}^{\ast}} i_j a_{\ii}=0,  {\quad j = 1,\ldots,d-1}
			\end{cases},
		\end{equation*}
		whose coefficient matrix is $A_d = (R_1//R_2//A_{\I^*_{d-1}})$, thus the assert.
		
		Case 2: $d$ even. We have $M_d = m_d = \frac{d}{2}$ thus the two equations in the system in~\eqref{point2} are equal and become:
		\begin{equation*}
			\sum_{\ii \in \I_{d-1}^{d/2}} a_{\ii} =0.
		\end{equation*}
		Thus, $\aa$ is a solution of the linear system 
		\begin{equation*}
			\begin{cases}
				\sum_{\ii \in \I_{d-1}^{d/2}} a_{\ii} =0\\
				\sum_{\ii \in \I_{d-1}^{\ast}} i_j a_{\ii}=0,  {\quad j = 1,\ldots,d-1}
			\end{cases},
		\end{equation*}
		whose coefficient matrix is $A_d = (\boldsymbol{1}_{n^*_d}//A_{\I^*_{d-1}})$, thus the assert.
	\end{proof}

	\begin{proof}[Proof of Proposition~\ref{characterization_minCX}]
		For $\ii \in \mathcal{X}_{d-1}$ such that $a_{\ii} = 0$, the type-0 pmf has $f(\ss_{\ii}) = f(\boldsymbol{1}_d - \ss_{\ii}) = 0$. 
		Therefore, if $a_{\ii} = 0$, the sum of the components of a Bernoulli random vector with the type-0 as pmf has not support on $\sum_{j=1}^{d-1}i_j$ or $d-\sum_{j=1}^{d-1}i_j$.
		From Point~\ref{null_coeff} of Theorem~\ref{thm:polyminCX}, $P^{\ast}(\zz) = \sum_{\ii \in \I^{\ast}_{d-1}} a_{\ii} \zz^{\ii}$.
		Therefore, $a_{\ii} = 0$ for every $\ii \notin \I^{\ast}_{d-1}$ and the type-0 pmf $\ff^{\ast}$ has support only on $M_d$ or $m_d$. 
		Thus $\ff^{\ast}$ is a $\Sigma_{cx}$-smallest element in $\BS$.
		Suppose $\ff$ is a $\Sigma_{cx}$-smallest pmf corresponding to a polynomial $Q(\zz)$ equivalent to $P^{\ast}(\zz)$. 
		From Proposition~\ref{prop:inverseMap}, we have that
		\begin{equation*}
			\ff = \lambda \ff^{\ast} + (1-\lambda) \ff^K.
		\end{equation*}
		Since $f(\xx) = f^{\ast}(\xx) = 0$ if $\xx \notin \mathcal{X}^{\ast}_d$, we have that also $\ff^K$ is a $\Sigma_{cx}$-smallest pmf.
		Conversely, if $\ff$ is such that $\ff = \lambda \ff^{\ast} + (1-\lambda) \ff^{K\ast}$, where $\ff^{K\ast}$ is a $\Sigma_{cx}$-smallest pmf with null polynomial, then also $\ff$ is a $\Sigma_{cx}$-smallest element in $\BS$.
	\end{proof}

	\section{The class $\BS$: examples and complements} \label{app:ExamplesComplements}
	
	\begin{example} \label{ex:polynomial_d3}
		We consider the case $d=3$, $\mathcal{SB}_3$. We have $\boldsymbol{a} = Q \ff$, where 
		\begin{equation*}
			\boldsymbol{a} =
			\begin{pmatrix}
				a_{00} \\
				a_{10} \\
				a_{01} \\
				a_{11} \\
			\end{pmatrix},
			\quad
			Q =
			\begin{pmatrix}
				1 & 0 & 0 & 0 & 0 & 0 & 0 & -1 \\
				0 & 1 & 0 & 0 & 0 & 0 & -1 & 0 \\
				0 & 0 & 1 & 0 & 0 & -1 & 0 & 0 \\
				0 & 0 & 0 & 1 & -1 & 0 & 0 & 0 \\
			\end{pmatrix},
			\quad
			\ff =
			\begin{pmatrix}
				f_{000} \\
				f_{100} \\
				f_{010} \\
				f_{110} \\
				f_{001} \\
				f_{101} \\
				f_{011} \\
				f_{111} \\
			\end{pmatrix}.
		\end{equation*}
		Therefore, for $d=3$, the polynomials in $\mathcal{C}_{\mathcal{H}}\subseteq \mathcal{I}_{\mathcal{P}}$ are of the form $P(\zz) = a_{00} + a_{10} z_1 + a_{01} z_2 + a_{11} z_1 z_2$, where
		\begin{equation*}
			\begin{cases}
				a_{00} = f_{000} - f_{111} \\
				a_{10} = f_{100} - f_{011} \\
				a_{01} = f_{010} - f_{101} \\
				a_{11} = f_{110} - f_{001} \\
			\end{cases}.
		\end{equation*}
		The polynomials in $\III_{\PPP}$ vanish at points $\zz = \boldsymbol{1}_{d-1} = (1,1)$, $\zz = \boldsymbol{1}_{d-1}^{-1} = (-1,1)$ and $\zz = \boldsymbol{1}_{d-1}^{-2} = (1,-1)$.
		For example, let us consider the following pmfs and their corresponding polynomials $\mathcal{H}(\ff)$:
		\begin{equation*}
			\begin{split}
				&\ff_1=(0.3,0.1,0.1,0,0.1,0,0,0.4), \\
				&\ff_2=(0.1,0.1,0.1,0.2,0.3,0,0,0.2), \\
				&\ff_3=(0,0.25,0.25,0,0.25,0,0,0.25), \\
				&\ff_4=(0.25,0,0,0.25,0,0.25,0.25,0),
			\end{split}        
			\qquad
			\begin{split}
				& P_1(\zz) = -0.1+0.1 z_1 + 0.1z_2 - 0.1z_1z_2; \\
				& P_2(\zz) = -0.1+0.1 z_1 + 0.1z_2 - 0.1z_1z_2; \\
				& P_3(\zz) = -0.25+0.25 z_1 + 0.25 z_2 - 0.25 z_1z_2; \\    
				& P_4(\zz) = 0.25 - 0.25 z_1 - 0.25 z_2 + 0.25 z_1z_2.
			\end{split}        
		\end{equation*}
		The polynomials $P_1(\zz)$ and $P_2(\zz)$ are identical, while $P_3(\zz) = \frac{5}{2} \cdot P_1(\zz)$ and $P_4(\zz) = -P_3(\zz)$. 
	\end{example}

	\begin{example}\label{ex:Ec-im}
		Consider the class $\BBB_3(2/5)$.
		The pmf $\rr = ( 0, \frac{1}{5},\frac{1}{5},\frac{1}{5},\frac{2}{5}, 0,0,0) $ is an extremal pmf of $\BBB_3(2/5)$ (see Example 1 in \cite{fontana2024high}) and it is not the type-0 pmf associated to its polynomial. 
		Indeed, we have $\mathcal{H}(\rr)= Q(\zz)= -\frac{1}{5}+ \frac{1}{5} z_1+\frac{1}{5}z_2-\frac{1}{5}z_1z_2$ and, the version of Algorithm 1 in \cite{fontana2024high} that holds for any $p \in [0,1] \cap \mathbb{Q}$, the type-0 pmf associated to $Q(\zz)$ is  $\ff=(0, \frac{3}{10}, \frac{3}{10},0, \frac{3}{10},0,0, \frac{1}{10})$.
		Thus $\rr$ is an extremal pmf and it is not a type-0 pmf.
		It is straightforward to see that we can write 
		\begin{equation*}
			\rr=\frac{2}{3}\ff+\frac{1}{3}\ff^K,
		\end{equation*}
		where $\frac{2}{3}\ff$ is the non-normalized type-0 pmf associated to $Q(\zz)$ and $\frac{1}{3}\ff^K$ is an element of the kernel of the linear relation defined by the matrix $Q$ in~\eqref{eq:poly_coeff} between $\RR^{2^{d}}$ and $\RR^{2^{d-1}}$. 
		However, it can be verified that $\ff^K$ has negative components and, therefore, it is not a pmf.
	\end{example}	
	
	Proposition~\ref{prop:MatrixSameRank} states an interesting property of the linear system in~\eqref{eq:linear_system}. 
	
	\begin{proposition} \label{prop:MatrixSameRank}
		If $d$ is odd, we have $ \rank(A_d) = \rank(A_{d+1}) = d$.
	\end{proposition}
	\begin{proof}
		Let $d$ be odd.
		Since $n^*_d = n^*_{d+1}$, it is clear that $A_d = (R_1//R_2//A_{\I^*_{d-1}})$ and $A_{d+1}=(\boldsymbol{1}_{n^*_{d+1}}^{\top}//A_{\I^*_{d}}) = (\boldsymbol{1}_{n^*_d}^{\top}//A_{\I^*_{d}})$ are matrices of order $(d+1) \times n^*_d$. 
		We now prove that $A_{\I^*_{d}}=(A_{\I^*_{d-1}}//R_1)$.
		We have $\ii_{d-1}\in \I^*_{d-1}$ if and only if $\sum_{k=1}^{d-1}i_k=M_d$ or $\sum_{k=1}^{d-1}i_k=m_d$ and $\ii_d\in \I^*_{d}$ if and only if $\sum_{k=1}^di_k=\frac{d+1}{2}=m_d$. 
		Let $\ii_{d-1}\in \I^*_{d-1}$. 
		If $\sum_{k=1}^{d-1} i_{d-1,k} =M_d$, then $(\ii_{d-1},1) \in \I^*_{d}$, if $\sum_{k=1}^{d-1} i_{d-1,k} =m_d$, than $(\ii_{d-1},0)\in \I^*_{d}$. 
		Therefore, $A_{\I^*_{d}}=(A_{\I^*_{d-1}}//R_1)$, where $R_1 \in \mathcal{M}(1\times n^*_d)$ with ones in correspondence of the indexes $\ii$ with sum $M_d$ and zeros elsewhere.
		Thus, $A_d=(R_1//R_2//A_{\I^*_{d-1}})$ and $A_{d+1}=(\boldsymbol{1}_{n^*_{d}}//A_{\I^*_{d-1}}//R_1)$. 
		Since $\boldsymbol{1}_{n^*_{d}}=R_1+R_2$ by construction, then 
		\begin{equation}\label{eqrnk}
			\rank(A_d) = \rank(A_{d+1}).
		\end{equation}
		Let $r_d:=\rank(A_d) = \rank(A_{d+1}) $. 
		We first observe that $r_d\leq d$, indeed the sum of the elements of the columns of $A_{\I^*_{d}}$ is $\frac{d+1}{2}$, thus $\boldsymbol{1}_{n^*_{d}}$ is a linear combination of the other rows.
		We prove that $r_d = d$ by induction. 
		From Example~\ref{X_d=3} and Example~\ref{X_d=4} we have that $r_3=3$. 
		Let assume that $r_{d} = d$.
		We know that $r_{d+2} \leq d+2$ and we want to prove that the equality holds. 
		From~\eqref{eqrnk} we have $\rank(A_{d+1}) = d$.
		The matrix $A_{d+1}=(\boldsymbol{1}_{n^*_{d}}//A_{\I^*_{d}})$ and the columns of $A_{\I^*_{d}}$ have exactly $m_d = (d+1)/2$ ones.
		Let $B_{d+1}$ be a submatrix extracted from $A_{d+1}$ by choosing $d$ linearly independent columns of $A_{d+1}$. 
		We have $B_{d+1}=( \boldsymbol{1}_{d}//\tilde{A}_{d+1})$, where $\tilde{A}_{d+1}$ is a submatrix of $A_{\I^*_{d}}$.
		Define $B_{d+2}=(\boldsymbol{1}_{d}//\boldsymbol{0}_{d}//\boldsymbol{0}_{d}//\tilde{A}_{d+1})$, where $\boldsymbol{0}_{d}$ is a row vector of all zeros.
		Let $\tilde{A}_{d+2}=(\boldsymbol{0}_{d}//\tilde{A}_{d+1})$. 
		This is a submatrix of $A_{\I^*_{d+2}}$. 
		Since $d+2$ is odd, the sum of the elements of the columns of $\tilde{A}_{d+2}$ is $m_d=M_{d+2}$.
		We have that $B_{d+2}$ is a submatrix of $A_{d+2}$ and $r_{d+2} = \rank(A_{d+2})\geq d$.
		For $d \geq 3$, we can always find other two columns in $A_{d+2}$ as follows: $(1, 0 , 1, \ii_1)$, where {$\ii_1 \in \mathcal{X}_{d}$} with $\sum_{k=1}^{d}i_{1,k} = M_{d+2}-1$ and $(0,1,0,\ii_2)$ where $\ii_2 \in \mathcal{X}_{d}$ with $\sum_{k=1}^{d}i_{2,k}=m_{d+2}$. 
		These two columns are independent from $B_{d+2}$ thus $r_{d+2} = \rank(A_{d+2}) \geq d+2$. 
		Thus, $r_{d+2} = d+2$ and we have the assert.
	\end{proof}

	We conclude with some examples. 
	Example~\ref{ex:build_matrix} illustrates Proposition~\ref{prop:MatrixSameRank}, while in Example~\ref{ex:X_d=5} and Example~\ref{ex:X_d=6} we find the $\Sigma_{cx}$-smallest elements in dimension $d=5$ and $d=6$, respectively.
	
	\begin{example} \label{ex:build_matrix}
		Let $d=5$. 
		In this example we compare the matrices $A_5$ and $A_6$ and their rank. 
		According to Corollary~\ref{corollary_system}, we build the following matrices:
		\begin{equation*}
			A_5 = 
			\begin{pmatrix}
				1 & 1 & 1 & 0 & 1 & 1 & 0 & 1 & 0 & 0 \\
				0 & 0 & 0 & 1 & 0 & 0 & 1 & 0 & 1 & 1 \\
				1 & 1 & 0 & 1 & 1 & 0 & 1 & 0 & 1 & 0 \\
				1 & 0 & 1 & 1 & 0 & 1 & 1 & 0 & 0 & 1 \\
				0 & 1 & 1 & 1 & 0 & 0 & 0 & 1 & 1 & 1 \\
				0 & 0 & 0 & 0 & 1 & 1 & 1 & 1 & 1 & 1 \\
			\end{pmatrix},
			\qquad
			A_6 = 
			\begin{pmatrix}
				1 & 1 & 1 & 1 & 1 & 1 & 1 & 1 & 1 & 1 \\
				1 & 1 & 0 & 1 & 1 & 0 & 1 & 0 & 1 & 0 \\
				1 & 0 & 1 & 1 & 0 & 1 & 1 & 0 & 0 & 1 \\
				0 & 1 & 1 & 1 & 0 & 0 & 0 & 1 & 1 & 1 \\
				0 & 0 & 0 & 0 & 1 & 1 & 1 & 1 & 1 & 1 \\
				1 & 1 & 1 & 0 & 1 & 1 & 0 & 1 & 0 & 0 \\
			\end{pmatrix}.
		\end{equation*}
		As one can see $A_5 = (R_1//R_2//A_{I_{4}^*})$ and $A_6 = (\boldsymbol{1}_{10}//A_{I_{4}^*}//R_1)$. 
		The rows of $A_6$ are linear combinations of the rows of $A_5$ and $\rank(A_5) = \rank(A_6) = 5$.
		Notice that the resulting matrix $A_6$ is not ordered according to the reverse-lexicographical criterion: however, the order of the columns is irrelevant for the purpose of Proposition~\ref{prop:MatrixSameRank}.
	\end{example}

	\bibliographystyle{plain}
	\bibliography{Biblio}
	
\end{document}